\documentclass{amsart}

\usepackage[english]{babel}
\usepackage[utf8]{inputenc}
\usepackage[T1]{fontenc}

\usepackage{amssymb,amsmath,amsfonts,bm,csquotes,enumitem,graphicx,hyperref,mathtools,mathrsfs,stmaryrd,tikz}
\usepackage{hyperref}
\usepackage{cleveref}

\DeclareMathOperator{\dens}{\mathrm{dens}}
\DeclareMathOperator{\Per}{\mathrm{Per}}
\DeclareMathOperator{\Aper}{\mathrm{Aper}}

\newcommand{\bsigma}{\boldsymbol{\sigma}}
\newcommand{\btau}{\boldsymbol{\tau}}

\newcommand{\N}{\mathbb{N}}
\newcommand{\Z}{\mathbb{Z}}

\newcommand{\cA}{\mathcal{A}}
\newcommand{\cB}{\mathcal{B}}
\newcommand{\cC}{\mathcal{C}}

\newcommand{\cI}{\mathcal{I}}

\newcommand{\cL}{\mathcal{L}}

\newcommand{\cP}{\mathcal{P}}

\newcommand{\cR}{\mathcal{R}}
\newcommand{\cS}{\mathcal{S}}
\newcommand{\cT}{\mathcal{T}}

\newtheorem{theo}{Theorem}
\newtheorem{cor}[theo]{Corollary}
\newtheorem{prop}[theo]{Proposition}
\newtheorem{lem}[theo]{Lemma}
\newtheorem{claim}[theo]{Claim}

\theoremstyle{remark}
\newtheorem{rema}[theo]{Remark}

\def\epsilon{\varepsilon}

\def\NN{\mathbb{N}}
\def\ZZ{\mathbb{Z}}

\def\SS{\texorpdfstring{$\mathcal{S}$}{}}
\def\PP{\mathcal{P}}
\def\TT{\mathcal{T}}

\def\coinc{\mathrm{coinc}}
\def\Per{\operatorname{Per}}
\def\pimeq{\pi_{{\rm meq}}}

\newcommand{\inttt}[1]{\lvert \kern-0.25ex \lvert \kern-0.25ex \lvert #1 \rvert \kern-0.25ex \rvert \kern-0.25ex \rvert}
\newcommand{\nor}[1]{\lvert \kern-0.25ex \lvert #1 \rvert \kern-0.25ex \rvert}
\title{The Jacobs--Keane theorem from the $\cS$-adic viewpoint}

\author{Felipe Arbulú}
\address{
Laboratoire Amiénois de Mathématique Fondamentale et Apliquée,
CNRS-UMR 7352,
Université de Picardie Jules Verne,
33 rue Saint Leu,
80039 Amiens cedex 1,
France.}
\email{felipe.arbulu@u-picardie.fr}

\author{Fabien Durand}
\address{
Laboratoire Amiénois de Mathématique Fondamentale et Apliquée,
CNRS-UMR 7352,
Université de Picardie Jules Verne,
33 rue Saint Leu,
80039 Amiens cedex 1,
France.}
\email{fabien.durand@u-picardie.fr}

\author{Bastián Espinoza}
\address{
Laboratoire Amiénois de Mathématique Fondamentale et Apliquée,
CNRS-UMR 7352,
Université de Picardie Jules Verne,
33 rue Saint Leu,
80039 Amiens cedex 1,
France.}
\email{bespinoza@dim.uchile.cl}

\date{\today}

\begin{document}

% Keywords

% Abstract
\begin{abstract}
In the light of recent developments of the ${\mathcal S}$-adic study of subshifts, we revisit, within this framework, a well-known result on Toeplitz subshifts due to Jacobs--Keane giving a sufficient combinatorial condition to ensure discrete spectrum.
We show that the notion of coincidences, originally introduced in the '$70$s for the study of the discrete spectrum of substitution subshifts, together with the ${\mathcal S}$-adic structure of the subshift allow to go deeper in the study of Toeplitz subshifts. 
We characterize spectral properties of the factor maps onto the maximal equicontinuous topological factors by means of coincidences density.
We also provide an easy to check necessary and sufficient condition to ensure unique ergodicity for constant length ${\mathcal S}$-adic subshifts.
\end{abstract}

\maketitle

%%%%%%%%%%%%%%%%%%%%%%%%%%%%%%%%%%%%%%%%%%%%%%%%%%
\section{Introduction}\label{section:intro}
%%%%%%%%%%%%%%%%%%%%%%%%%%%%%%%%%%%%%%%%%%%%%%%%%%

A series of recent works shows that the underlying $\cS$-adic structures of zero-entropy subshifts shed new light on the study of their properties as shown by \cite{BSTY19, BCBD+21, DDMP21, EM22,Esp22, GLL22, AD23, Esp23}.  
This approach, which goes back to \cite{LV92, Fer96}, consists in finding a representation of the subshift in terms of an infinite sequence of morphisms defined on finitely generated monoids.
For more references on this subject, see the general references \cite{BD14, DP22}.

In the present work we investigate the spectral properties of Toeplitz subshifts through the $\cS$-adic perspective.
For a sequence $x = (x_n)_{n \in \ZZ} \in \cA^\ZZ$ on some finite alphabet $\cA$ and $p \ge 0$, we denote by $\Per_{p}(x)$ the set of integers $n$ such that $(x_{n+kp})_{k \in \ZZ}$ is constant.
The aperiodic part of $x$ is the set
\[
\Aper(x)
= \ZZ \setminus \bigcup_{p \ge 0} \Per_p(x).
\]
A Toeplitz sequence is a non periodic sequence $x$ such that $\Aper(x) = \emptyset$ and a Toeplitz subshift is the subshift generated by a Toeplitz sequence $x$, {\em i.e.}, the shift orbit closure of $x$.

Toeplitz sequences $x$ always have a {\em periodic structure}, that is, a sequence $(p_n)_{n \ge 0}$ such that for all $n \ge 0$ we have
\begin{itemize}
    \item $\Per_{p_n}(x) \not= \emptyset$ and $\Per_{p_n}(x) \not= \Per_q(x)$ for all $0 \le q < p_n$;
    \item $p_n$ divides $p_{n+1}$; and
    \item $\cup_{n \ge 0} \Per_{p_n}(x) = \ZZ$.
\end{itemize}
This results, topologically, in a maximal equicontinuous topological factor which is the odometer \cite{Wil84}.
A famous result of Jacobs and Keane \cite{JK69, Mar75,Dow05} gives a sufficient condition by means of this periodic structure to ensure that the Toeplitz subshift has {\em discrete spectrum}, that is, there exists a measure theoretical isomorphism between the subshift and a translation on a compact Abelian group.

In this paper we revisit this result through the $\cS$-adic framework.
Indeed from \cite[Theorem 8]{GJ00} (see also \cite[Theorem 6.1.7]{DP22}) we know that Toeplitz subshifts are generated by proper, constant length, primitive and recognizable $\cS$-adic sequences. 
We will precise this can be relaxed changing ``proper'' by ``having coincidences'' in \Cref{prop:toeplitzSadic}.

Observe that when a minimal dynamical system $(X,T)$ continuously factorizes onto a group rotation then the factor map is unique up to a translation.
Thus below we will speak about ``the'' factor map instead of ``a'' factor map.
We denote by $\pimeq$ this continuous factor map.

In \cite[Theorem 6]{JK69} the authors show that for a Toeplitz subshift the equivalent conditions in \Cref{item:JK_1} and \Cref{item:JK_2} below imply the discrete spectrum of the system.
However, these conditions are not necessary for the discrete spectrum.
Indeed, they are equivalent to the strictly stronger condition that the subshift is a {\em regular extension} of its maximal equicontinuous topological factor (the union of the singleton $\pimeq$-fibers has full measure).
For more details see \Cref{subsec:top_dyn}.
There exist discrete spectrum Toeplitz subshifts that do not fulfill this condition (see \Cref{{sec:examples}}).

The difference between regularity and having discrete spectrum has been studied in detail in \cite{DG16,Gar17,HLSY21,GJY21} for topological dynamical systems.
There are given various strictly different notions, from topological to more measurable ones, all implying the discrete spectrum, such as mean equicontinuity or $\mu$-mean equicontinuity (which is equivalent to the discrete spectrum relative to $\mu$). 

In our work we introduced the point of view of coincidences to propose a more computable way to see the discrete spectrum. 
This allows us to build many examples having various spectral and ergodic properties. 

The notion of coincidence is quantified below by the term ``$\coinc$'' that is defined in \Cref{subsec:coincidences}.
Coincidences have shown to be very relevant to detect the pure discrete spectrum.
For example, in the context of constant length substitution subshifts \cite{Dek78}, two-symbol Pisot substitutions \cite{BD02, HS03} or more general substitutions \cite{Liv87, Que87}.
Moreover, this notion has been intensively used to tackle the Pisot substitution conjecture \cite{CS01,Pyth02,BK06,BBK06}.

For a subset $A \subseteq \ZZ$, the {\em density} of $A$ is $\dens(A) = \limsup_{N \to +\infty} \frac{|A \cap [0, N)|}{N}$.

\begin{theo}
\label{theo:charac_MEF_is_regular}
Let $(X,S)$ be the subshift generated by the Toeplitz sequence $x$ with periodic structure $(p_n)_{n \ge 0}$.
Then, $X$ is defined by some constant length, recognizable directive sequence with coincidences $\btau = (\tau_n)_{n \ge 0}$ and the following are equivalent:
\begin{enumerate}
    \item  \label{item:MEF_1}
    $(X, S)$ is a regular extension of its maximal equicontinuous topological factor: the union of the singleton $\pimeq$-fibers has full measure.
    \item \label{item:JK_1} $\dens(\Per_{p_n}(x)) \to 1$ as $n \to \infty$.
    \item \label{item:JK_2} $\dens(\Aper(y)) = 0$ for every $y \in X$.
    \item \label{item:MEF_2} For every $n \ge 0$, we have
    \begin{equation*}
    \lim_{N \to +\infty} \frac{|\coinc (\tau_{[n,N)})|}{|\tau_{[n,N)}|}
    = 1.
    \end{equation*}
    \item \label{item:MEF_3} We have
    \begin{equation*}
    \lim_{N \to +\infty} \frac{|\coinc (\tau_{[0,N)})|}{|\tau_{[0,N)}|}
    = 1.
    \end{equation*}
    
    \item \label{item:MEF_4} There is a contraction $\btau' = (\tau'_k)_{k \ge 0}$ of $\btau$ such that 
    \begin{equation*}
    \sum_{k \ge 0} \frac{|\coinc (\tau'_k)|}{|\tau'_k|}
    = +\infty.
    \end{equation*}
\end{enumerate}
Moreover, any of these conditions implies that $(X, S)$ is uniquely ergodic and that $\pimeq$ is a measure theoretical isomorphism between $(X, S)$ and its maximal equicontinuous topological factor. 
\end{theo}

It is observed in \cite[Theorem 4.12]{GJY21}[Theorem 54]{Gar17} that for a minimal topological dynamical system the regularity condition is equivalent to the {\em diam-mean equicontinuity property}.

The $\cS$-adic approach also allows us to study measure-theoretic aspects of a subshift.
This is done by considering the notion of {\em adapted} directive sequences.
Roughly speaking, when considering adapted sequences we get rid of parts of the system that are negligible.

We present a variation of \Cref{theo:charac_MEF_is_regular} in terms of coincidence densities relative to the non negligible part of the system.
This result is stated for subshifts generated by constant length directive sequences.
It includes the case of Toeplitz subshifts as shown by \Cref{prop:toeplitzSadic}, but also  cases of non Toeplitz, non uniquely ergodic subshifts.

\begin{theo}\label{theo:charac_MEF_is_iso}
Let $X$ be generated by the constant length, recognizable and primitive directive sequence $\btau = (\tau_n : \cA_{n+1}^\ast \to \cA_n^\ast)_{n \ge 0}$.
Suppose that $\btau$ is $(\cA'_n)_{n \ge 0}$-adapted to an ergodic measure $\mu$ of $X$.
The following conditions are equivalent.
\begin{enumerate}
    \item \label{item:iso_1} \label{cond:theo:charac_MEF_is_iso:is_iso}
    The map $\pimeq : (X,S,\mu) \to (\ZZ_{(|\tau_{[0,n)}|)_{n \ge 0}}, +1, \nu)$ defines a measure-theoretic isomorphism.
    
    \item \label{item:iso_2} \label{cond:theo:charac_MEF_is_iso:limit_coinc_mu}
    For every $n \geq 0$, we have
    \begin{equation*}
    \limsup_{N \to +\infty} \frac{|\coinc_{\cA'_N}(\tau_{[n,N)})|}{|\tau_{[n,N)}|}
    = 1.
    \end{equation*}
\end{enumerate}
\end{theo}

We illustrate below, as applications of our results, various behaviours of Toeplitz subshifts with respect to discrete spectrum.
They are detailed in \Cref{sec:examples}.

Let $\cA = \{ a,b \}$ and consider the directive sequence $\btau = (\tau_n : \cA^\ast \to \cA^\ast)_{n \ge 0}$ such that for all $n \ge 0$ the morphism $\tau_n : \cA^* \to \cA^*$ is defined by
\[
a \mapsto a^{n^2} b a a^{n^2}, \enskip
b \mapsto b^{n^2} a a b^{n^2}.
\]
The Toeplitz subshift $(X_{\btau} , S)$ is not a regular extension of its maximal equicontinuous topological factor, it has two ergodic measures $\mu_a$ and $\mu_b$.
Moreover, the map $\pimeq$ defines a measure theoretical isomorphism to the maximal equicontinuous topological factor for both $(X_{\btau}, S, \mu_a)$ and $(X_{\btau}, S, \mu_b)$ (see \Cref{subsec:example_2}).

Now let $\cA = \{a, b, c\}$ and consider the directive sequence $\btau = (\tau_n : \cA^\ast \to \cA^\ast)_{n \ge 0}$ such that for all $n \ge 0$ the morphism $\tau_n : \cA^* \to \cA^*$ is defined by
\[
a \mapsto (ab)^{n^2+1} ac (ab)^{n^2+1}, \enskip 
b \mapsto (ab)^{n^2+1} bc (ab)^{n^2+1}, \enskip
c \mapsto c (ab)^{n^2} c cc  c (ab)^{n^2} c.
\]
The Toeplitz subshift $(X_{\btau} , S)$ is not a regular extension of its maximal equicontinuous topological factor, it possesses a unique invariant probability measure $\mu$ such that $\pimeq$ is a measure theoretical isomorphism to its maximal equicontinuous topological factor (see \Cref{subsec:example_3}).
As far as we know, the first such example was given in \cite[Section 5]{DK15}.

Other interesting examples are the following.
Let $\cA = \{a, b, c\}$, $m : \NN \to \NN$ defined by $2 m(n) + 3 = 3^n$ for each $n \ge 0$ and consider the directive sequence $\btau = (\tau_n : \cA^\ast \to \cA^\ast)_{n \ge 0}$ such that for all $n \ge 0$ the morphism $\tau_n : \cA^* \to \cA^*$ is
\[
a \mapsto (ab)^{m(n)} a b c, \enskip
b \mapsto a(ab)^{m(n)} a c, \enskip
c  \mapsto (ab)^{m(n) } c b c.
\]
The Toeplitz subshift $(X_{\btau} , S)$ is such that its maximal equicontinuous topological factor is $(\ZZ_3, +1)$.
Moreover it has a unique ergodic measure $\mu$.
Nevertheless $\pimeq$ is not a measure theoretical isomorphism, but $(X_{\btau} , S, \mu)$ has discrete spectrum: it is measure theoretically isomorphic to $(\ZZ_3 \times \ZZ /2\ZZ, +(1,1))$ (see \Cref{sec:ex4}).

Let $\cA = \{1, 2, a, b, c\}$, $s : \NN \to \NN$ a nonnegative function such that $2s(n)+5 = 3^n$ for $n \ge 0$ and consider the directive sequence $\btau = (\tau_n : \cA^\ast \to \cA^\ast)_{n \ge 0}$ such that for all $n \ge 0$ the morphism $\tau_n : \cA^* \to \cA^*$ is defined by
\[
\begin{array}{ll}
1 \mapsto a (12)^{s(n)} 1 2 b c,  & a \mapsto (ab)^{s(n)} a b 1 2 c\\
2 \mapsto a(12)^{s(n)} 2 1 b c, & b \mapsto  a(ab)^{s(n)} b 1 2 c\\
& c \mapsto (ab)^{s(n)} a b 1 2 c.
\end{array}
\]

The Toeplitz subshift $(X_{\btau} , S)$ has two ergodic measures $\mu$ and $\nu$ and its maximal equicontinuous topological factor is $(\ZZ_3, +1)$.
The map $\pimeq$ is a measure theoretical isomorphism for the measure $\mu$ whereas it is not for the measure $\nu$.
Nevertheless, the system $(X_{\btau}, S, \nu)$ has discrete spectrum: it is measure theoretically isomorphic to $(\ZZ_3 \times \ZZ /2\ZZ, +(1,1))$ (see \Cref{sec:ex5}).

Let $\cA = \{1, 2\}$, $s : \NN \to \NN$ a nonnegative function such that $3s(n)+1 = 5^{2n}$ for $n \ge 0$ and consider the directive sequence $\btau = (\tau_n : \cA^\ast \to \cA^\ast)_{n \ge 0}$ such that for all $n \ge 0$ the morphism $\tau_n : \cA^* \to \cA^*$ is defined by
\[
1 \mapsto (121)^{s(n)} 2, \enskip
2 \mapsto 1(121)^{s(n)}.
\]
The Toeplitz subshift $(X_{\btau}, S)$ has a unique ergodic measure $\mu$.
However, $(X_{\btau}, S, \mu)$ does not have discrete spectrum (see \Cref{sec:ex6}).

Finally we consider an example of a minimal non Toeplitz subshift generated by a constant length directive sequence that has positive entropy and where the conclusion of \Cref{theo:charac_MEF_is_iso} holds (see \Cref{sec:ex7}).

\section*{Acknowledgement}
The first author is supported by ANID (ex CONICYT) Doctorado Becas Chile $72210185$ grant. The third author thanks Doctoral Fellowship CONICYT-PFCHA Doctorado Nacional $2020-21202229$.
All authors are supported by ANR-22-CE40-0011 project Inside Zero Entropy Systems.

%%%%%%%%%%%%%%%%%%%%%%%%%%%%%%%%%%%%%%%%%%%%%%%%%%
\section{Background}
%%%%%%%%%%%%%%%%%%%%%%%%%%%%%%%%%%%%%%%%%%%%%%%%%%

\subsection{Basics in topological dynamics}\label{subsec:top_dyn}

A {\em topological dynamical system} (or just a system) is a pair $(X, T)$ where $X$ is a compact metric space and $T : X \to X$ is a homeomorphism.
The system $(X, T)$ is {\em minimal} if for every point $x \in X$ the orbit $\{T^n x : n \in \ZZ\}$ is dense in $X$.

Let $(X, T)$ and $(Y, S)$ be two topological dynamical systems.
We say that $(Y, S)$ is a {\em topological factor} of $(X, T)$ if there exists a continuous and onto map $\pi : X \to Y$ such that
\begin{equation}\label{eq:top_factor}
\pi \circ T
= S \circ \pi.
\end{equation}
In this case, we say that $\pi$ is a \emph{factor map} and that $(X, T)$ is an {\em extension} of $(Y, S)$.
If in addition the map $\pi$ in \eqref{eq:top_factor} is a homeomorphism, we say that it is a \emph{topological conjugacy} and that $(X, T)$ and $(Y, S)$ are \emph{topologically conjugate}.

Let $(X, T)$ be a uniquely ergodic system and denote by $\mu$ its unique invariant probability measure.
We say that $(X, T)$ is a {\em regular} extension of $(Y, S)$ if there exists a factor map $\pi : X \to Y$ such that
\begin{equation}\label{eq:regular_extension_1}
\mu(\{x \in X : |\pi^{-1}(\{\pi(x)\})| = 1\})
= 1.    
\end{equation}

Any topological dynamical system $(X, T)$ possesses a {\em maximal equicontinuous topological factor} \cite{EG60}, that is, there exists an equicontinuous system $(Y, S)$ and a factor map $\pi : X \to Y$ such that every other equicontinuous factor of $(X, T)$ is a factor of $(Y, S)$.
The system $(Y, S)$ is unique up to topological conjugacy.
Moreover, if $(X, T)$ is minimal, then $(Y, S)$ corresponds to a rotation on a compact abelian group and we usually denote by $\nu$ the Haar measure on $Y$.

In the minimal case, if $(X, T)$ is a regular extension of its maximal equicontinuous topological factor $(Y, S)$, then condition \eqref{eq:regular_extension_1} is equivalent to
\begin{equation}\label{eq:regular_extension_2}
\nu(\{y \in Y : |\pi^{-1}(\{y\})| = 1\})
= 1.
\end{equation}

Moreover, since any two such factor maps differ only by a translation, condition \eqref{eq:regular_extension_2} is independent of the choice of the factor map $\pi$ and of the representation of $(Y,S)$ in its conjugacy class.

\subsection{Subshifts}\label{subsec:subshifts}

Let $\cA$ be a finite and nonempty set that we call {\em alphabet}.
Elements in $\cA$ are called {\em letters} or {\em symbols}.
The number of letters of $\cA$ is denoted by $|\cA|$.
The set of finite sequences or {\em words} of length $\ell \in \NN$ with letters in $\cA$ is denoted by $\cA^\ell$.
The {\em full shift} $\cA^\ZZ$ is the set of all bi-infinite sequences $(x_n)_{n \in \ZZ}$ with $x_n \in \cA$ for all $n \in \ZZ$.

A word $w = w_0 w_1 \ldots w_{\ell - 1} \in \cA^\ell$ can be seen as an element of the free monoid $\cA^\ast$ endowed with the operation of concatenation (whose neutral element is $\epsilon$, the empty word).
The integer $\ell$ is the {\em length} of the word $w$ and is denoted by $|w| = \ell$; the length of the empty word is $0$.
We say that a word $w$ {\em occurs} in a sequence $x=(x_n)_{n \in \ZZ} \in X$ if there exists $m \in \ZZ$ such that $w = x_m \cdots x_{m+|w|-1}$.
We use the same notion for words.
For such a word $w$ in $\cA^\ast$ and a letter $a$ in $\cA$, the number of occurrences of $a$ in $w$ is denoted by $|w|_a$.
A nonempty word $w = w_0 w_1 \ldots w_{\ell - 1} \in \cA^\ast$ \emph{starts} (resp. \emph{ends}) with a nonempty word $u \in \cA^\ast$ if $u = w_0 \ldots w_{i - 1}$ for some $i \le \ell$ (resp. $u = w_j \ldots w_{\ell - 1}$ for some $j \ge 0$).

The {\em shift map} $S : \cA^\ZZ \to \cA^\ZZ$ is defined by $S ((x_n)_{n \in \ZZ}) = (x_{n+1})_{n \in \ZZ}$.
A {\em subshift} is a subset $X$ of a fullshift $\cA^\ZZ$ which is closed for the product topology and invariant under the shift map.
Thus $(X, S)$ is a topological dynamical system that we also call subshift.

Let $(X, S)$ be a subshift.
For $x \in X$ and $i, j \in \ZZ$ with $i < j$ we define $x_{[i, j)} = x_i x_{i+1} \ldots x_{j-1}$.
The \emph{language} of $(X, S)$ is the set $\cL(X)$ containing all words $w \in \cA^\ast$ such that $w = x_{[m, m + |w|)}$ for some $x = (x_n)_{n \in \ZZ} \in X$ and $m \in \ZZ$.
In this case, we also say that $w$ is a \emph{factor} (also called \emph{subword}) of $x$.
Given $x \in X$, the language $\cL(x)$ is the set of all words that occur in $x$.
For two words $u,v \in \cL(X)$, the \emph{cylinder set} $[u.v]$ is the set $\{x \in X : x_{[-|u|,|v|)} = u v\}$.
When $u$ is the empty word we only write $[v]$, erasing the dot.
We remark that cylinder sets are clopen sets and they form a base for the topology of the subshift.

\subsection{Morphisms and substitutions}
\label{subsec:morphisms}

By a {\em morphism} we simply mean a morphism $\tau : \cA^\ast \to \cB^\ast$ between the free monoids $\cA^\ast$ and $\cB^\ast$ for some finite alphabets $\cA$ and $\cB$.
We say that $\tau$ is {\em erasing} whenever there exists $a$ in $\cA$ such that $\tau(a)$ is the empty word.
Otherwise we say it is {\em nonerasing}.

The morphism $\tau$ is {\em positive} if for all $a$ in $\cA$ every letter of $\cB$ occurs in $\tau(a)$.
The morphism $\tau$ is {\em proper} if there exist $p$ and $s$ in $\cB$ such that $\tau(a)$ starts with $p$ and ends with $s$ for every letter $a$ in $\cA$. 
The morphism $\tau$ is of {\em constant length} if the length of the word $\tau(a)$ does not depend on $a$.
In that case, we denote by $|\tau|$ the length of the word $\tau(a)$ for any $a$ in $\cA$.
Observe that if $\tau : \cA^\ast \to \cB^\ast$ and $\sigma : \cB^\ast \to \cC^\ast$ are morphisms of constant length, then the morphism $\sigma \circ \tau$ is of constant length and
\[
|\sigma \circ \tau|
= |\sigma| |\tau|.
\]

When it is nonerasing, $\tau$ extends naturally to maps from $\cA^\NN$ to itself and from $\cA^\ZZ$ to itself in the obvious way by concatenation (in the case of a two-sided sequence we apply $\tau$ to positive and negative coordinates separately and we concatenate the results at the coordinate zero).

To the morphism $\tau : \cA^\ast \to \cB^\ast$ we associate a {\em composition matrix} $M_\tau$ indexed by $\cB \times \cA$ such that its entry at position $(b, a)$ is the number of occurrences of $b$ in $\tau(a)$ for every $a \in \cA$ and $b \in \cB$.

\subsection{Coincidences of constant length morphisms}
\label{subsec:coincidences}

Throughout this article we will use the notion of {\em coincidences} of constant length morphisms.
Let $\tau : \cA^* \to \cB^*$ be a constant length morphism. 
A {\it coincidence} of $\tau$ relative to $\cA' \subseteq \cA$ is an integer $i \in [0, |\tau|)$ such that the map that sends a letter $a$ in $\cA'$ to the $i$-th letter of $\tau(a)$ is constant.
We denote by $\coinc_{\cA'}(\tau)$ the set of such integers.
When $\cA'=\cA$ we just say these integers are coincidences and denote the set of such integers by $\coinc (\tau)$.

The following lemma states how coincidences behave when morphisms are composed.
The elementary proof is left to the reader.

\begin{lem} \label{lem:inclusion_coinc}
Let $\tau : \cA^\ast \to \cB^\ast$ and $\sigma : \cB^\ast \to \cC^\ast$ be two morphisms of constant length.
Then the set $\coinc(\sigma \circ \tau)$ contains the disjoint union
\begin{align*}
&\enskip \{i |\sigma| + j : i \in \coinc(\tau),\ 0 \le j < |\sigma|\}\\
&\cup \{i |\sigma| + j : i \in [0, |\tau|) \setminus \coinc(\tau),\ j \in \coinc(\sigma)\}
\end{align*}
and we have
\begin{equation} \label{eq:ineq_coinc_morphisms}
1 - \frac{|\coinc(\sigma \circ \tau)|}{|\sigma \circ \tau|}
\le \left(1 - \frac{|\coinc(\sigma)|}{|\sigma|}\right) \left(1 - \frac{|\coinc(\tau)|}{|\tau|}\right).
\end{equation}
Moreover, if $|\cA| = |\cB| = 2$ then inequality \eqref{eq:ineq_coinc_morphisms} is an equality between both terms.
\end{lem}

\subsection{\SS-adic subshifts}
\label{subsec:Sadicsubshifts}

A {\em directive sequence} $\btau = (\tau_n : \cA_{n+1}^\ast \to \cA_n^\ast)_{n \ge 0}$ is a sequence of morphisms.
From now on, we only consider morphisms that are nonerasing.
When all morphisms $\tau_n$ for $n\ge 1$ are proper we say that $\btau$ is {\em proper}; and when all morphisms $\tau_n$ for $n\ge 1$ are of constant length we say that $\btau$ is of {\em constant length}.

For $0 \le n < N$, we denote  by  $\tau_{[n,N)}$ or $\tau_{[n,N-1]}$ the morphism $\tau_n \circ \tau_{n+1} \circ \dots \circ \tau_{N-1}$.
We say $\btau$ is {\em primitive} if for any $n \in \NN$ there exists $N > n$ such that $M_{\tau_{[n,N)}} > 0$, {\em i.e.}, all letters in $\cA_{n}$ occur in $\tau_{[n,N)}(a)$ for all $a\in \cA_{N}$.

For $n \in \N$, the {\em language $\cL^{(n)}({\btau})$ of level $n$ associated with $\btau$} is defined by 
\[
\cL^{(n)}({\btau})
= \big\{w \in \cA_n^\ast : \mbox{$w$ occurs in $\tau_{[n,N)}(a)$ for some $a \in \cA_N$ and $N > n$}\big\}
\]
and $X_{\btau}^{(n)}$ is the set of points $x \in \cA_n^\ZZ$ such that $\cL(x) \subseteq \cL^{(n)}({\btau})$.
When nonempty, this set clearly defines a subshift 
that we call the {\em subshift generated by $\cL^{(n)}({\btau})$}.
We set $X_{\btau} = X_{\btau}^{(0)}$ and call $(X_{\btau}, S)$ the {\em $\cS$-adic subshift} generated by the directive sequence $\btau$.
See \cite[Chapter 6]{DP22} for more details.

A \emph{contraction} of $\btau = (\tau_n : \cA_{n+1}^\ast \to \cA_n^\ast)_{n \ge 0}$ is a directive sequence of the form
\[
\btau'
= (\tau'_k = \tau_{[n_k, n_{k+1})} : \cA_{n_{k+1}}^\ast \to \cA_{n_k}^\ast)_{k \ge 0},
\]
where the sequence $(n_k)_{k \ge 0}$ is such that $n_0 = 0$ and $n_k < n_{k+1}$ for all $k \ge 0$.
Observe that any contraction of $\btau$ generates the same $\cS$-adic subshift $X_{\btau}$.

Inspired by the definition of substitutions with coincidences in \cite{Dek78}, we say a constant length directive sequence $\btau$ {\em has coincidences} if there is a contraction $\btau' = (\tau_k')_{k \ge 0}$ of $\btau$ such that $\coinc(\tau_k') \not= \emptyset$ for all $k \ge 0$.

One says that $\btau$ has {\em finite alphabet rank} when
$\liminf_{n\to+\infty} |\cA_n| < +\infty$.
This notion is related to the {\em topological rank} defined more generally for minimal Cantor systems.
We refer to \cite{DDMP21} for more details about their interaction.

\subsection{Recognizability and dynamical partitions}\label{subsec:reco}

We now present the recognizability property for morphisms and directive sequences in terms of topological partitions.
The usual definition and a more general discussion can be found in \cite{BSTY19, BPR23, BPRS23}.

Let $\tau : \cA^\ast \to \cB^\ast$ be a nonerasing morphism and let $X \subseteq \cA^\ZZ$ be a subshift.
We say that $\btau$ is {\em recognizable} on $X$ if
\[
\PP
= \{S^k \tau([a]) : a \in \cA, \ 0 \le k < |\tau(a)|\}
\]
defines a partition of the subshift $\bigcup_{k \in \ZZ} S^k \tau(X)$.

We say that a directive sequence $\btau = (\tau_n : \cA_{n+1}^\ast \to \cA_n^\ast)_{n \ge 0}$ is recognizable if
\begin{equation}\label{eq:rep_sadic}
\PP_n
= \{S^k \tau_{[0,n)}([a]) :
a \in \cA_n,\
0 \le k < |\tau_{[0,n)}(a)|\}, \enskip
n \ge 0,
\end{equation}
defines a sequence of partitions of $X_{\btau}$.
We have that $\btau$ is recognizable if and only if for each $n \ge 0$ the morphism $\tau_n$ is recognizable on $X_{\btau}^{(n+1)}$ \cite[Section 4]{BSTY19}.

It will be convenient, in order to manipulate these partitions, to consider the following definitions.
For $n \ge 0$ and $a \in \cA_n$, we denote by
\begin{equation}
\label{def:towers}
\TT_n(a)
= \bigcup_{0 \le k < |\tau_{[0,n)}(a)|} S^k \tau_{[0,n)}([a])
\end{equation}
the {\em tower} indexed by $a$ and by $B_n(a)$ the {\em base} $\tau_{[0,n)}([a])$ of this tower.
The {\em base of $\PP_n$} is
\begin{equation} \label{eq:bases}
B_n
= \bigcup_{a \in \cA_n} B_n(a).
\end{equation}
We make here two important observations about $(\PP_n)_{n\geq0}$.
For any $0 \le n < m$,
\begin{enumerate}
    \item the partition $\PP_m$ is finer than the partition $\PP_n$; and
    \item we have that
    \begin{equation}\label{eq:counting_intersections_PP}
    |\{k\in[0,|\tau_{[0,m)}(b)|): S^kB_m(b)\subseteq B_n(a)\}|
    = |\tau_{[n,m)}(b)|_a,
    \enskip a\in\cA_n,\enskip b\in\cA_m.
    \end{equation}
\end{enumerate}

We now state a fundamental theorem for the ergodic and topological study of $\cS$-adic subshifts.

\begin{theo}[{\cite[Lemma 6.3]{BSTY19}}]\label{theo:reco}
Let $\btau = (\tau_n : \cA_{n+1}^\ast \to \cA_n^\ast)_{n \ge 0}$ be a recognizable directive sequence and $\mu$ an invariant measure of $(X_{\btau},S)$.
Then, for $\mu$-almost all $x \in X_{\btau}$, the set $\bigcap_{n \ge 0} S^{k_n} \tau_{[0,n)}([a_n])$ is a singleton, where $(k_n,a_n)_{n \ge 0}$ for $n \ge 0$ is the {\em $(\PP_n)_{n \ge 0}$-address of $x$}, that is, the unique sequence satisfying $x \in S^{k_n} \tau_{[0,n)}([a_n])$, $n \ge 0$.
\end{theo}

We end this section with the following remark: if in addition to recognizability the sequence $\btau$ is assumed to be primitive and proper, then the atoms of $\bigcup_{n \ge 0} \PP_n$ generate the topology of $X_{\btau}$ \cite[Proposition 2.2]{DL12}.

\subsection{Invariant measures}\label{subsec:invariant_measures}

Let $\btau = (\tau_n : \cA_{n+1}^\ast \to \cA_n^\ast)_{n \ge 0}$ be a recognizable directive sequence and $\mu$ be an invariant measure of $(X_{\btau}, S)$.
In this context, \Cref{theo:reco} implies the following.

\begin{cor}
\label{cor:sigmaalg}
The sequence $(\PP_n)_{n \ge 0}$ $\mu$-generates the Borel $\sigma$-algebra of $X_{\btau}$, that is, $\mu$ is uniquely determined by the values it assigns to $B_n(a)$, $a \in \cA_n$, $n \ge 0$.
\end{cor}

This motivates the following.
For $n \ge 0$ we define the column vector $\mu_n$ by
\[
\mu_n 
= (\mu_n(a) : a \in \cA_n), \enskip
\text{where} \enskip
\mu_n(a)
= \mu(B_n(a)).
\]

Since $\btau$ is recognizable, we have 
\begin{equation}
\label{eq:inv_measure}
\mu_m = M_{\tau_{[m,n)}} \mu_n, \enskip
0 \le m < n.
\end{equation}
The measure of the tower indexed by $a$ is
\begin{equation}\label{eq:measure_tower}
\mu(\TT_n(a)))
= \mu(B_n(a)) |\tau_{[0,n)}|
= \mu_n(a) |\tau_{[0,n)}|, \enskip
n \ge 0.
\end{equation}

\medskip

When dealing with measure theoretical arguments, it is natural to get rid of small portions of the space.
We will need the definition of {\em adapted} directive sequences to state our main results, where we distinguish those towers with small measure and those with a measure uniformly bounded away from zero.

For each $n \ge 0$ let $\cA'_n \subseteq \cA_n$ be a nonempty alphabet and set
\[
\TT'_n
= \bigcup_{a \in \cA'_n} \TT_n(a), \enskip
\TT'_{\geq n}
= \bigcap_{m \geq n} \TT'_{m}, \enskip
\TT'
= \bigcup_{n \geq 0} \TT'_{\geq n}.
\]
Observe that $\TT' = \liminf_{n \to +\infty} \TT'_n$.

We say that $\btau$ is {\em $(\cA'_n)_{n \ge 0}$-weakly-adapted to $\mu$} if
\begin{equation}
\label{eq:new_adapted}
\sum_{n \ge 0} \mu(X_{\btau}\setminus\TT'_n)
< +\infty.
\end{equation}
If in addition there exists $\delta > 0$ such that
\begin{equation}
\label{eq:new_delta}
\mu(\TT_n(a))
\ge \delta, \enskip a \in \cA'_n, \enskip n \ge 0 ,
\end{equation}
we say that $\btau$ is {\em $(\cA'_n)_{n \ge 0}$-adapted to $\mu$} (for the constant $\delta$).
Observe that in such case $\sup_{n \ge 0} |\cA'_n |$ is bounded by $1/\delta$.

When $\btau$ has finite alphabet rank, up to a contraction and a renaming of the alphabets, there exists $\cA_\mu \subseteq \cA_n$ for all $n \ge 0$ such that $\btau$ is $(\cA_\mu)$-adapted to $\mu$.

The previous definitions are inspired by the definition of {\em clean Bratteli diagrams} in \cite{BDM10} (see also \cite[Section 3]{BKMS13}).

In the weakly-adapted case, the Borel--Cantelli lemma implies that
\begin{equation} \label{eq:measure_TT'}
\mu(\TT') = 1.
\end{equation}

\subsection{Convergence of frequencies}
It is a well known fact that for substitutions systems the measures $\mu(B_n(a))$ correspond to the letter frequencies of the iterates of the substitution \cite[Chapter 5]{Que87}.
A similar result was proved in \cite[Proposition 5.1]{BKMS13} in the more general context of uniquely ergodic Bratteli--Vershik systems.

In this paper, we will need an extension of these results that is valid for the class of non uniquely ergodic systems we later consider.

\begin{prop}\label{prop:mu_generic_and_freqs}
Let $\btau = (\tau_n : \cA_{n+1}^\ast \to \cA_n^\ast)_{n \ge 0}$ be a primitive and recognizable directive sequence, let $\mu$ an ergodic measure for $(X_{\btau},S)$ and let $(b_m)_{m \ge 0}$ a sequence of letters with $b_m$ in $\cA_m$.
Suppose that any of the following conditions hold:
\begin{enumerate}
    \item \label{item:mu_generic_1}
    there exists a $\mu$-generic point $x$ in $X_{\btau}$ such that $x \in \TT_m(b_m)$ for all $m \ge 0$;
    
    \item \label{item:mu_generic_2}
    $\liminf_{m \to +\infty} \mu(\TT_m(b_m)) > 0$;
    
    \item \label{item:mu_generic_3}
    $\btau$ is $(\cA'_n)_{n \ge 0}$-adapted to $\mu$ and $b_m$ belongs to $\cA'_m$ for all $m \ge 0$;
    
    \item \label{item:mu_generic_4}
    the system $(X_{\btau}, S)$ is uniquely ergodic.
\end{enumerate}
Then, for any $n\geq0$ and $a\in\cA_n$, we have
\begin{equation}
\label{eq:mu_generic_and_freqs}
\lim_{m\to+\infty}
\frac{|\tau_{[n,m)}(b_m)|_a}{|\tau_{[0,m)}(b_m)|}
= \mu(B_n(a)).
\end{equation}
\end{prop}

\begin{proof}
Let $x$ be a $\mu$-generic point such that $x$ belongs to $\TT_m(b_m)$ for all $m \ge 0$.
We define $i_m = \max\{k\le0: S^k x\in B_m\}$ and $j_m = \min\{k>0: S^kx\in B_m\}$.
Observe that $j_m- i_m = |\tau_{[0,m)}(b_m)|$.
Since $\btau$ is recognizable, from \eqref{eq:counting_intersections_PP} we get, for $n<m$,
\[
|\{k\in[0,|\tau_{[0,m)}(b_m)|): S^kB_m(b_m)\subseteq B_n(a)\}|
= |\tau_{[n,m)}(b_m)|_a, \enskip
b_m \in \cA_m.
\]
This implies, as $S^{i_m}x$ and $S^{j_m}x$ are in $B_m(b_m)$ and $\cP_m$ is finer than $\cP_n$, that
$$|\{k\in[i_m,j_m):S^kx\in B_n(a)\}| = |\tau_{[n,m)}(b_m)|_a. $$
The left-hand side corresponds to the ergodic sum $\sum_{k\in[i_m,j_m)} 1_{B_n(a)}(S^kx)$ of the indicator function $1_{B_n(a)}$. 
Hence, since $x$ is $\mu$-generic, we conclude that
$$
\lim_{m\to+\infty} 
\frac{|\tau_{[n,m)}(b_m)|_a}{|\tau_{[n,m)}(b_m)|} = 
\lim_{m\to+\infty} 
\frac{1}{j_m - i_m}
\sum_{k\in[i_m,j_m)} 1_{B_n(a)}(S^kx) = \mu(B_n(a))
$$
and \eqref{eq:mu_generic_and_freqs} holds.

\smallskip

Now suppose that $\liminf_{m \to +\infty} \mu(\TT_m(b_m)) > 0$.
We assume, with the aim to obtain a contradiction, that there exist $n\geq0$ and $a\in\cA_n$ such that $\tfrac{|\tau_{[n,m)}(b_m)|_a}{|\tau_{[0,m)}(b_m)|}$ is bounded away from $\mu(B_n(a))$ for all $m$ in an infinite set $E\subseteq\N$.

Let 
\[
K
= \limsup_{\substack{m \to +\infty \\ m\in E}}
\cT_m(b_m).
\]
From our hypothesis, we have $\mu(K) > 0$ and hence that there is a $\mu$-generic point $x$ in $K$.
We have that $x$ is in $\cT_m(b)$ for all $m$ belonging to an infinite set $E' \subseteq E$.
But then from \Cref{item:mu_generic_1} in \Cref{prop:mu_generic_and_freqs} the sequence $\Bigl(\tfrac{|\tau_{[n,m)}(b_m)|_a}{|\tau_{[0,m)}(b)|} : m \in E'\Bigr)$ is arbitrarily close to $\mu(B_n(a))$, which contradicts our previous assumption.
This shows that \eqref{eq:mu_generic_and_freqs} holds.

\smallskip

Next, we suppose that $\btau$ is $(\cA'_n)_{n \ge 0}$-adapted to $\mu$ and $b_m \in \cA'_m$ for all $m \ge 0$.
Then, from \eqref{eq:new_delta} we get $\liminf_{m \to +\infty} \mu(\TT_m(b_m)) > 0$ and thus \eqref{eq:mu_generic_and_freqs} holds by \Cref{item:mu_generic_2} of this proposition.

\smallskip

Finally, if the system $(X_{\btau}, S)$ is uniquely ergodic, then any point is $\mu$-generic and \eqref{eq:mu_generic_and_freqs} holds for any sequence of letters $b_m$ in $\cA_m$ from \Cref{item:mu_generic_1} in \Cref{prop:mu_generic_and_freqs}.
\end{proof}

%%%%%%%%%%%%%%%%%%%%%%%%
\subsection{Toeplitz sequences as \SS-adic subshifts}
%%%%%%%%%%%%%%%%%%%%%%%%

In \cite[Theorem 8]{GJ00} it is shown that the class of Toeplitz subshifts coincides, up to topological conjugacy, with the class of expansive Bratteli--Vershik systems defined on diagrams with the {\em equal path number property}.
This is the starting point of the link between Toeplitz subshifts and $\cS$-adic subshifts that is summarized in the proposition below.
This will be used throughout the article.
We follow the proof in \cite[Theorem 6.1.7]{DP22}.

\begin{prop}
\label{prop:toeplitzSadic}
Let $(X, S)$ be a subshift.
The following are equivalent.
\begin{enumerate}
    \item \label{item:toeplitzSadic_1}
    $(X, S)$ is a Toeplitz subshift;
    \item \label{item:toeplitzSadic_2}
    $(X, S)$ is generated by a constant length, primitive, proper and recognizable directive sequence; and
    \item \label{item:toeplitzSadic_3}
    $(X, S)$ is generated by a constant length, primitive and recognizable directive sequence with coincidences. 
\end{enumerate}
\end{prop}

\begin{proof}
It is clear that \Cref{item:toeplitzSadic_2} implies \Cref{item:toeplitzSadic_3}.
We now prove that \Cref{item:toeplitzSadic_3} implies \Cref{item:toeplitzSadic_1}.
Let $\btau = (\tau_n : \cA_{n+1}^\ast \to \cA_n^\ast)_{n \ge 0}$ be a constant length, primitive, proper, and recognizable directive sequence.
For each $n \ge 0$, let $k_n$ be a coincidence of $\tau_n$, $\ell_n = \lfloor |\tau_{[0,n)}| / 2 \rfloor$, $r_n = |\tau_{[0,n)}| - \ell_n$ and $x$ a point belonging to $\cap_{n \ge 0} \tau_{[0,n)}([\cA_n . \cA_n])$.
If we define the sequence of points $x_n = S^{k_n |\tau_{[0,n)}| + \ell_n} x$, $n \ge 0$, we observe that
\[
(x_n)_{[-\ell_n, r_n)}
= (x_n)_{[- \ell_n + k \cdot |\tau_{[0, n+1)}|, r_n + k \cdot |\tau_{[0, n+1)}|)}, \enskip
k \in \ZZ.
\]
This implies that any accumulation point of the sequence $(x_n)_{n \ge 0}$ is a Toeplitz sequence, and hence $(X_{\btau}, S)$ is a Toeplitz subshift.

It remains to prove that \Cref{item:toeplitzSadic_1} implies \Cref{item:toeplitzSadic_2}.

\begin{claim}
Let $x\in \cA^\Z$ be a Toeplitz sequence.
Then there exists a nonempty alphabet $\cB$, a constant length and positive morphism with coincidences $\tau : \cB^\ast \to \cA^\ast$ and a Toeplitz sequence $y \in \cB^\ZZ$ such that $x = \tau(y)$.
Moreover, the morphism $\tau$ is recognizable on the subshift $Y$ generated by $y$.  
\end{claim}

\begin{proof}
Let $m \in \NN$ be large enough so that all letters that occur in $x$ also occur in $x_{[-m, m)}$.
Define $u = x_{[-m,0)}$ and $v = x_{[0,m)}$.
As $x$ is a Toeplitz sequence, there exists a minimal $p > 2m$ such that
\[
x_{[-m + kp, m + kp)}
= x_{[-m, m)}
= uv, \enskip
k \in \ZZ.
\]
Define $\cB'$ be the set of words $\{x_{[kp, (k+1)p)} : k \in \ZZ\} \subseteq \cA^*$.
Let $\tau : \cB \to \cB'$ be a bijection, where $\cB$ is an alphabet.
The map $\tau$ of course defines a morphism from $\cB^*$ to $\cA^*$.

We have that $\tau$ is positive, with coincidences and $|\tau| = p$.
It is clear that there exists a unique point $y \in \cB^\ZZ$ such that $x = \tau(y)$.
Since $x$ is Toeplitz, we have that $y$ is Toeplitz.
We are left to prove that $\tau$ is recognizable on the subshift $Y$ generated by $y$.

Consider the cylinder $U = [u.v]$ in $X$ and define
\[
\Per_q(x', U)
= \{n \in \ZZ : S^{n + kq}(x') \in U,\ k \in \ZZ\}, \enskip
x' \in X, \enskip
q \in \NN.
\]
Since $p$ is minimal we have $\Per_q(x, U) = \emptyset$ if $q < p$ and $\Per_p(x, U) \not= \emptyset$.
The set $\{x' \in X : \Per_p(x', U) \not= \emptyset\}$ is nonempty, closed and $S$-invariant, hence it is equal to $X$ as $(X, S)$ is minimal.
From this we see that if
\[
C
= \{x' \in X : \Per_p(x', U) = \Per_p(x, U)\}
\]
then $\{S^i C : 0 \le i < p\}$ is a clopen partition of $X$.
Finally, since $(X, S)$ is minimal one has $C = \bigcup_{b \in \cB} \tau([b])$, where the cylinders $[b]$, $b \in \cB$, are considered in $Y$.
This implies that
\[
\{S^k \tau([b]) : b \in \cB, \ 0 \le k < p\}
\]
is a clopen partition of $X$, and thus $\tau$ is recognizable on $(Y, S)$, finishing the proof of the claim.
\end{proof}
By the claim, we inductively construct a sequence of alphabets $(\cB_n)_{n \ge 0}$ with $\cB_0 = \cA$; a sequence of Toeplitz sequences $(y_n)_{n \ge 0}$ with $y_n \in \cB_n^\ZZ$ for $n \ge 0$ such that $y_0 = x$; and a directive sequence of morphisms $\btau = (\tau_n : \cB_{n+1}^\ast \to \cB_n^\ast)_{n \ge 0}$ which is of constant length, primitive, with coincidences and such that for all $n \ge 0$ the morphism $\tau_n$ is recognizable on the subshift generated by $y_n$.
Since we have $y_n = \tau_{[n, N)}(y_N)$ for all $0 \le n < N$, we see that $y_n$ generates the subshift $X_{\btau}^{(n)}$.
In particular, $\btau$ is recognizable and $X = X_{\btau}$, which finishes the proof.
\end{proof}

%%%%%%%%%%%%%%%%%%%%%%%%%%%%%%%%%%
\subsection{The maximal equicontinuous topological factor}
\label{subsec:meq}
%%%%%%%%%%%%%%%%%%%%%%%%%%%%%%%%%%

\subsubsection{Odometers}

For a sequence of positive integers $(p_n)_{n \ge 0}$ such that $p_n$ divides $p_{n+1}$ for $n \ge 0$, the \emph{odometer} given by $(p_n)_{n \ge 0}$ is the system $(\ZZ_{(p_n)_{n \ge 0}}, +1)$, where
\[
\ZZ_{(p_n)_{n \ge 0}}
= \varprojlim \ZZ / p_n \ZZ
= \bigg\{ (x_n)_{n \ge 0} \in \prod_{n \ge 0} \ZZ / p_n \ZZ : x_{n+1} \equiv x_n \pmod{p_n},\ n \ge 0 \bigg\}
\]
and the map ``$+1$'' is given by
\[
(x_n)_{n \ge 0}
\mapsto (x_n + 1 \pmod{p_n})_{n \ge 0}.
\]
If $p_n = p^n$ for some $p \ge 1$ and all $n \ge 0$, we simply denote $Z_{(p_n)_{n \ge 0}}$ by $\ZZ_p$.

Odometers are, in particular, minimal rotations on compact abelian groups.
They are equicontinuous and uniquely ergodic, where the unique ergodic measure is the Haar measure.
In particular, the measurable system given by the odometer has discrete spectrum.

For $\ell \ge 0$ and $z = (z_n)_{n \ge 0} \in \ZZ_{(p_n)_{n \ge 0}}$ we define the cylinder
\[
[z_0, z_1, \ldots, z_{\ell - 1}]
= \{(x_n)_{n \ge 0} \in \ZZ_{(p_n)_{n \ge 0}} : x_k = z_k,\ 0 \le k < \ell\}.
\]
The Haar measure $\nu$ on $(\ZZ_{(p_n)_{n \ge 0}}, +1)$ satisfies
\begin{equation}\label{eq:haar}
\nu([z_0, z_1, \ldots, z_{\ell - 1}])
= \frac{1}{p_0 p_1 \ldots p_{\ell - 1}}.
\end{equation}

We finally remark that the odometer is {\em coalescent} as a measure theoretical system, that is, every measurable automorphism of the odometer is an isomorphism.
This follows since it has discrete spectrum \cite{HP68}.

\subsubsection{The maximal equicontinuous topological factor}

Let $\btau = (\tau_n)_{n \ge 0}$ be a constant length, primitive, and recognizable directive sequence with coincidences.
It is a well-known fact that the maximal equicontinuous topological factor of $(X_{\btau}, S)$ corresponds to the odometer $(\ZZ_{(|\tau_{[0,n)}|)_{n \ge 0}}, +1)$.
Indeed, this follows from \Cref{prop:toeplitzSadic} and \cite[Theorem 2.2]{Wil84}.

The factor map $\pimeq : X_{\btau} \to \ZZ_{(|\tau_{[0,n)}|)_{n \ge 0}}$ can be described as follows.
Let $x \in X_{\btau}$.
By recognizability of $\btau$, for every $n \ge 0$ there exists $k_n(x)$ with $0 \le k_n(x) < |\tau_{[0,n)}|$ such that $x
\in S^{k_n(x)} B_n$, where $B_n$ is the base given by \eqref{eq:bases}.
Then we define
\begin{equation}\label{eq:pieq}
\pimeq(x)
= (k_n(x))_{n \ge 0}.
\end{equation}
It can be observed that $k_{n+1}(x) \equiv k_n(x) \pmod{|\tau_{[0,n)}|}$ for $x \in X_{\btau}$.
In order to help the reader, it is important to notice the following relation
\begin{equation}\label{eq:cylinder_odometer}
\pimeq^{-1}([z_0, z_1, \ldots, z_{\ell}])
= S^{z_{\ell}} B_\ell, \enskip
(z_n)_{n \ge 0} \in \ZZ_{(|\tau_{[0,n)}|)_{n \ge 0}}, \enskip
\ell \ge 0.
\end{equation}

The following lemma will be useful.
\begin{lem} \label{lem:measure_cylinders}
Let $0 \le n < m$ and $A_{n,m}$ be a subset of $[0, |\tau_{[n,m}|)$.
Then
\[
\nu(\{z \in \ZZ_{(|\tau_{[0,n)}|)_{n \ge 0}} : z_m \in z_n + |\tau_{[0,n)}| A_{n,m}\})
= \frac{|A_{n,m}|}{|\tau_{[n,m)}|}.
\]
\end{lem}

\begin{proof}
Observe that, from \eqref{eq:haar}, the set $\{z \in \ZZ_{(|\tau_{[0,n)}|)_{n \ge 0}} : z_m \in z_n + |\tau_{[0,n)}| A_{n,m}\}$ is a disjoint union of $|\tau_{[0,n)}| \cdot |A_{n,m}|$ cylinders of measure $1 / |\tau_{[0,m)}|$.
We obtain
\[
\nu(\{z \in \ZZ_{(|\tau_{[0,n)}|)_{n \ge 0}} : z_m \in z_n + |\tau_{[0,n)}| A_{n,m}\})
= \frac{|A_{n,m}|}{|\tau_{[n,m)}|}.
\]
\end{proof}

%%%%%%%%%%%%%%%%%%%%%%%%%%%%%%%%%%%%%%%%%%%%%%%%%%%%%%%%%%%%%%%%
\section{About unique ergodicity of \SS-adic subshifts of constant length}
%%%%%%%%%%%%%%%%%%%%%%%%%%%%%%%%%%%%%%%%%%%%%%%%%%%%%%%%%%%%%%%%

In this section we give a necessary and sufficient condition for unique ergodicity in the specific case of $\cS$-adic subshifts generated by constant length and recognizable directive sequences. 
This condition is in terms of the matrices of the directive sequence and can be check relatively easily.

\begin{lem} \label{lem:series_lemma}
Let $(a_{n, N} : 0 \le n < N)$ be a doubly indexed sequence in $[0,1]$ that satisfies
\begin{equation} \label{eq:ineq_product}
a_{l, n}
\le a_{l, m} a_{m, n}, \enskip
l < m < n.
\end{equation}
The following are equivalent.
\begin{enumerate}
    \item \label{item:series_1}
    \[
    \lim_{N \to +\infty} a_{n, N}
    = 0, \enskip
    \text{for all $n \ge 0$};
    \]
    
    \item \label{item:series_2}
    there exists an increasing sequence $(n_k)_{k \ge 0}$ such that
    \[
    \lim_{K \to +\infty} a_{n_k, n_K}
    = 0, \enskip
    \text{for all $k \ge 0$; and}
    \]
    
    \item \label{item:series_3}
    there exists an increasing sequence $(n_k)_{k \ge 0}$ such that
    \[
    \sum_{k \ge 0} a_{n_k, n_{k+1}}
    < +\infty.
    \]
\end{enumerate}
\end{lem}

\begin{proof}
We first prove that \Cref{item:series_1} is equivalent to \Cref{item:series_2}.
It is clear that \Cref{item:series_1} implies \Cref{item:series_2}.
Let now $n \ge 0$ and $(n_k)_{k \ge 0}$ be such that \Cref{item:series_2} holds.
From \eqref{eq:ineq_product}, we deduce that the sequence $(a_{n, N} : N > n)$ is nonincreasing in $N$.
In particular, since it is bounded, we have $\lim_{N \to +\infty} a_{n, N}
= \inf_{N > n} a_{n, N}$.
Then, for any integers $0 \le k < K$ such that $n < n_k < n_K$, \Cref{eq:ineq_product} yields
\[
0
\le \inf_{N > n} a_{n, N}
\le a_{n, n_K}
\le a_{n_k, n_K}.
\]
By letting $K \to +\infty$, we obtain \Cref{item:series_1}.

\smallskip

\Cref{item:series_3} can be obtained from \Cref{item:series_2} by considering a sequence $(n_k)_{k \ge 0}$ such that $a_{n_k, n_{k+1}} \le 2^{-k}$ for $k \ge 0$.
It is left to prove that \Cref{item:series_3} implies \Cref{item:series_2}.
Let $k \ge 0$.
By using \eqref{eq:ineq_product} and the inequality $\log(x) \le x - 1$ for $x \in [0,1]$, we obtain
\[
a_{n_k, n_K}
\le \prod_{k \le j < K} a_{n_j, n_{j+1}}
\le \exp \Biggl(\sum_{k \le j < K} a_{n_j, n_{j+1}} -K+k\Biggr), \enskip
0 \le k < K,
\]
from which we deduce \Cref{item:series_2} if we let $K \to +\infty$.
\end{proof}

We now prove the main result of this section.
Denote by $|v|$ the sum of the entries of a vector $v$.

\begin{theo}
\label{theo:CriteriaUniqErgo}
Let $\btau = (\tau_n: \cA_{n+1}^\ast \to \cA_n^\ast)_{n \ge 0}$ be a constant length and recognizable directive sequence.
The following are equivalent.
\begin{enumerate}
    \item \label{item:ue_1} There exists a contraction $\btau' = (\tau'_k = \tau'_{[n_k, n_{k+1})})_{k \ge 0}$ of $\btau$ such that the vectors $(v_k)_{k \ge 0}$ given by
    \[
    v_k(a)
    = \min_{b \in \cA_{n_{k+1}}} |\tau'_k(b)|_a, \enskip
    a \in \cA_{n_k}, \enskip
    k \ge 0,
    \]
    satisfy
    \begin{equation} \label{eq:series_ue}
    \sum_{k \ge 0} \frac{|v_k|}{|\tau'_k|}
    = +\infty.
    \end{equation}
    
    \item \label{item:ue_2} The system $(X_{\btau}, S)$ is uniquely ergodic.
\end{enumerate}
\end{theo}

\begin{proof}
Let $\btau' = (\tau'_k = \tau_{[n_k, n_{k+1})} : \cA_{n_{k+1}}^\ast \to \cA_{n_k}^\ast)_{k \ge 0}$ be a contraction of $\btau$ and $(v_k)_{k \ge 0}$ be the vectors such that \Cref{item:ue_1} in \Cref{theo:CriteriaUniqErgo} holds.
Observe that, with this definition of the vectors $(v_k)_{k \ge 0}$, we have
\begin{equation} \label{eq:v_k}
M_{\tau'_k}
= v_k 1_{k+1} + M'_k, \enskip k \ge 0,
\end{equation}
where the matrix $M'_k$ is nonnegative and $1_{k+1}$ is the row vector of ones in $\ZZ^{\cA_{n_{k+1}}}$.

\begin{claim}
Define the doubly indexed sequences of vectors $(v_{k, K} : 0 \le k < K)$ and of matrices $(M'_{k, K} : 0 \le k < K)$ inductively by $v_{k, k+1} = v_k$; $M'_{k, k+1} = M'_k$ for $k \ge 0$; and
\begin{align*}
v_{k,K+1}
&= |v_K| v_{k,K} + M'_{k,K} v_K + (|\tau'_K| - |v_K|) v_{k,K} \\
M'_{k,K+1}
&= M'_{k,K} M'_K, \enskip
K > k.
\end{align*}
Then we have
\begin{equation} \label{eq:v_k_K}
M_{\tau'_{[k,K)}}
= v_{k,K} 1_K + M'_{k,K}, \enskip
0 \le k < K.
\end{equation}
\end{claim}

\begin{proof}
From \eqref{eq:v_k} we have $1_k M_k' = (|\tau'_k| - |v_k|) 1_{k+1}$ for all $k \ge 0$.
We prove the claim by induction on $K$.
\Cref{eq:v_k_K} holds if $K = k+1$ by \eqref{eq:v_k}.
If \eqref{eq:v_k_K} holds, then
\begin{align*}
M_{\tau'_{[k,K+1)}}
&= (v_{k,K} 1_K + M'_{k,K})(v_K 1_{K+1} + M'_K)\\
&= (|v_K| v_{k,K} + M'_{k,K} v_K + (|\tau'_K| - |v_K|) v_{k,K}) 1_{K+1}
+ M'_{k,K} M'_K\\
&= v_{k,K+1} 1_{K+1} + M'_{k,K+1},
\end{align*}
thus proving the claim by induction.
\end{proof}

Observe that from the previous claim we have
\[
1_k M'_{k,K}
= (|\tau'_{[k,K)}| - |v_{k,K}|) 1_K
\]
and hence
\begin{align*}
|v_{k,K+1}|
&= |v_K| |v_{k,K}| + (1_k M'_{k,K}) v_K + (|\tau'_K| - |v_K|) |v_{k,K}|\\
&= |v_K| |v_{k,K}| + (|\tau'_{[k,K)}| - |v_{k,K}|) |v_K| + (|\tau'_K| - |v_K|) |v_{k,K}|\\
&= (|\tau'_K| - |v_K|) |v_{k,K}| + |\tau'_{[k,K)}| |v_K|, \enskip
0 \le k < K.
\end{align*}
After rearranging terms, we obtain
\[
1 - \frac{|v_{k,K+1}|}{|\tau'_{[k,K+1)}|}
= \Biggl(1 - \frac{|v_{k,K}|}{|\tau'_{[k,K)}|}\Biggr)
\Biggl(1 - \frac{|v_K|}{|\tau'_K|} \Biggr), \enskip
0 \le k < K,
\]
therefore
\[
1 - \frac{|v_{k, K}|}{|\tau'_{k,K}|}
= \prod_{k \le \ell < K}
\left(1 - \frac{|v_\ell|}{|\tau'_\ell|}\right), \enskip
0 \le k < K.
\]

Define
\[
a_{k, K}
= 1 - \frac{|v_{k, K}|}{|\tau'_{k,K}|}, \enskip
0 \le k < K.
\]
Since \Cref{eq:series_ue} holds, \Cref{lem:series_lemma} applied to $(a_{k,K} : 0 \le k < K)$ gives
\begin{equation} \label{eq:limit_v_k}
\lim_{K \to +\infty} \frac{|v_{k,K}|}{|\tau'_{[k,K)}|}
= 1, \enskip
k \ge 0.
\end{equation}

We now prove that \Cref{item:ue_2} in \Cref{theo:CriteriaUniqErgo} holds.
It is enough to prove that $(X_{\btau'}, S)$ is uniquely ergodic.
Let $\mu$ be an invariant measure of $(X_{\btau'}, S)$ and define the sequence of column vectors $(\mu_k)_{k \ge 0}$ by
\[
\mu_k
= \mu(\tau'_{[0,k)}([a])), \enskip
a \in \cA_{n_k}, \enskip
k \ge 0.
\]
From \Cref{cor:sigmaalg}, it is enough to prove that
\begin{equation}\label{eq:limit_mu}
\mu_k
= \lim_{K \to +\infty} \frac{v_{k,K}}{|\tau'_{[0,K)}|}, \enskip
k \ge 0.
\end{equation}

Let $0 \le k < K$.
By recognizability of $\btau'$, we have that $|\mu_k| = |\tau'_{[0,k)}|^{-1}$ and $\mu_k = M_{\tau'_{[k, K)}} \mu_K$.
Hence, from \eqref{eq:v_k_K} we deduce
\[
\mu_k
= \frac{v_{k,K}}{|\tau'_{[0,K)}|} + M'_{k,K} \mu_K.
\]
If we multiply both sides by $|\tau'_{[0,k)}|$ and sum over all letters $a \in \cA_{n_k}$, we obtain
\[
\frac{|v_{k,K}|}{|\tau'_{[k,K)}|} + |\tau'_{[0,k)}| |M'_{k,K} \mu_K|
= 1.
\]
From \eqref{eq:limit_v_k}, we have $|M'_{k,K} \mu_K| \to 0$ as $K \to +\infty$ and, since the matrix $M'_{k, K}$ is nonnegative, we finally deduce \eqref{eq:limit_mu}.

\smallskip

Now assume that \Cref{item:ue_2} in \Cref{theo:CriteriaUniqErgo} holds.
Define $n_0 = 0$ and denote by $\mu$ the unique invariant probability measure of $(X_{\btau}, S)$.
From \Cref{prop:mu_generic_and_freqs} we inductively construct an increasing sequence $(n_k)_{k \ge 0}$ such that
\begin{equation} \label{eq:min_ue}
\frac{\min_{b \in \cA_{n_{k+1}}} |\tau_{[n_k, n_{k+1})}(b)|_a}{|\tau_{[n_k, n_{k+1})}|}
\ge \frac{1}{2} |\tau_{[0, n_k)}| \mu_{n_k}(a), \enskip
a \in \cA_{n_k}, \enskip
k \ge 0
\end{equation}
Define the contraction $\btau' = (\tau'_k = \tau'_{[n_k, n_{k+1})})_{k \ge 0}$ of $\btau$.
By summing over all letters $a \in \cA_{n_k}$ in \eqref{eq:min_ue}, we deduce that \eqref{eq:series_ue} holds for the contraction $\btau'$.
This shows that \Cref{item:ue_2} implies \Cref{item:ue_1} in \Cref{theo:CriteriaUniqErgo} and finishes the proof. 
\end{proof}

The following corollary shows that the conditions stated in \Cref{theo:charac_MEF_is_regular} implies the unique ergodicity.

\begin{cor}
Let $\btau = (\tau_n: \cA_{n+1}^\ast \to \cA_n^\ast)_{n \ge 0}$ be a constant length and recognizable directive sequence.
If there exists a contraction $\btau' = (\tau'_k)_{k \ge 0}$ of $\btau$ such that
\begin{equation} \label{eq:series_coinc}
\sum_{k \ge 0} \frac{|\coinc(\tau'_k)|}{|\tau'_k|}
= +\infty,
\end{equation}
then the system $(X_{\btau}, S)$ is uniquely ergodic.
\end{cor}

Let us make some observations about \Cref{theo:CriteriaUniqErgo}.
There are in the literature other such conditions implying unique ergodicity. 
Let $m_n$ and $M_n$ be the smallest and greatest entry of $M_{\tau_n}$, respectively. 
It comes from \cite{Sen81} that if
\[
\sum_{n \ge 0} \frac{m_n}{M_n}
= +\infty
\]
then $(X_{\btau}, S)$ is uniquely ergodic.
This implies the well-known fact that, when dealing with square matrices, not necessarily in the constant length case, if a matrix with positive coefficients occurs infinitely many times in the sequence $(M_{\tau_n} : n \ge 0)$, then $(X_{\btau}, S)$ is uniquely ergodic.
This criteria for unique ergodicity can be tracked down to \cite{Fur60}.

For the constant length case, we do not need all coefficients to be positive.

\begin{cor}
Let $\btau = (\tau_n: \cA_{n+1}^\ast \to \cA_n^\ast)_{n \ge 0}$ be a constant length and recognizable directive sequence.
Suppose that there exists a morphism $\tau : \cA^\ast \to \cB^\ast$ which satisfies:
\begin{enumerate}
    \item There exists a letter $b$ in $\cB$ which occurs in $\tau(a)$ for every $a \in \cA$; and
    \item We have $\tau_n = \tau$ for infinitely many values of $n$.
\end{enumerate}
Then $(X_{\btau}, S)$ is uniquely ergodic.
\end{cor}

Our result also implies unique ergodicity when matrices are upper or lower triangular with positive coefficients (in the corresponding triangular and diagonal parts of the matrices) and one of each type occurring infinitely many times.

\Cref{theo:CriteriaUniqErgo} covers more situations in the constant length case than classical criteria.
For example, one cannot deduce from the Seneta criteria the 
unique ergodicity when 
\[
M_{\tau_n}
= \begin{pmatrix}
1 & n^2 \\ 2n^2-1 & n^2
\end{pmatrix}, \enskip
n \ge 0
\]
whereas it is uniquely ergodic as $|v_n| = 1+n^2$.
Moreover, in our framework we do not require $(|\cA_n| : n \ge 0)$ to be bounded, which is a constraint in \cite{Sen81}.

Let us mention that when  
\[
M_{\tau_n}
= \begin{pmatrix}
a_n & 1 \\ 1 & a_n
\end{pmatrix}, \enskip
\text{where} \enskip
a_n > 0
\]
it is shown in \cite{FFT09} that $(X_{\btau}, S)$ is uniquely ergodic if and only if $\sum_n 1/a_n$ diverges.

We finally remark in this section that there exist other result about ergodic measures in the context of constant length directive sequences.
However, these results go in the direction of characterizing when the system possesses exactly $k$ ergodic measures, where $k$ is the rank of every composition matrix in the directive sequence \cite[Theorem 3.7]{ABKK17}.

%%%%%%%%%%%%%%%%%%%%%%%%%%%%%%%%%%%%%%%%%%%%%%%%%%%%%%%%%%%%
\section{Preliminary definitions and results}\label{section:prelim}
%%%%%%%%%%%%%%%%%%%%%%%%%%%%%%%%%%%%%%%%%%%%%%%%%%%%%%%%%%%%

In this section we include all main preliminary lemmas that allow us to deduce our main results in the next section.

Let $\btau = (\tau_n: \cA_{n+1}^\ast \to \cA_n^\ast)_{n \ge 0}$ be a constant length and recognizable directive sequence and let $\mu$ be an ergodic measure of $(X_{\btau}, S)$.
For each $n \ge 0$, let $\cA'_n \subseteq \cA_n$ be a nonempty alphabet.
For convenience, we set  
\[
\coinc'(\tau_{[n,N)}) 
=
\coinc_{\cA'_N}(\tau_{[n,N)}), \enskip
N > n
\] 
where $\coinc_{\cA'_N}(\tau_{[n,N)})$ is the set of coincidences of the morphism $\tau_{[n,N)}$ relatively to $\cA'_N$ introduced in \Cref{subsec:coincidences}.

From now on, we suppose that $\mu$ is $(\cA'_n)_{n \ge 0}$-adapted to $\mu$.
We recall the definition of the sequence of partitions $(\PP_n)_{n \ge 0}$ given in \Cref{subsec:reco} as well as the definition of the sets $B'_n$, $\TT'_n$, $\TT'_{\ge n}$ for $n \ge 0$ and $\TT'$ given in \Cref{subsec:invariant_measures}.

In the following we aim to control the set (an their measure) of $z \in \ZZ_{(|\tau_{[0,n)}|)_{n \ge 0}}$ having a unique $\pimeq$-preimage.
We set
\begin{align*}
\cI'
&= \{z \in \ZZ_{(|\tau_{[0,n)}|)_{n \ge 0}}: |\pimeq^{-1}(\{z\}) \cap \TT'|
\le 1\} \enskip \text{and}\\
\cI'_{m}
&= \{z \in \ZZ_{(|\tau_{[0,n)}|)_{n \ge 0}} : \text{$\pimeq^{-1}(\{z\}) \cap \TT'_{\ge m}$ is contained in some atom of $\PP_m$}\}
\end{align*}
for each $m \ge 0$.

We recall that $\nu$ is the Haar measure on $\ZZ_{(|\tau_{[0,n)}|)_{n \ge 0}}$.

\begin{lem}\label{lem:cI'_equals_bigcap_cI'n}
The set $\cI'$ is equal to $\bigcap_{n\geq0} \cI'_{n}$ up to a set of $\nu$-measure zero.
\end{lem}

\begin{proof}
Observe that the inclusion $\cI' \subseteq \bigcap_{n\geq0} \cI'_{n}$ follows directly from the definitions.
Thus, it is enough to prove that
\begin{equation}\label{eq:lem:cI'_equals_bigcap_cI'n:1}
\nu(\cap_{n\geq0} \cI'_n)
\leq \nu(\cI').
\end{equation}
We observe that if $z$ belongs to $\cI'_n$ for every $n \ge 0$, then all points in $\pimeq^{-1}(\{z\}) \cap \TT'$ share the same $(\PP_n)_{n \ge 0}$-address.
Therefore, by \Cref{theo:reco} we deduce
\[
\mu(\pimeq^{-1}(\cap_{n\geq0} \cI'_n))
\leq \mu(\pimeq^{-1}(\cI')),
\]
which implies \eqref{eq:lem:cI'_equals_bigcap_cI'n:1}.
\end{proof}

The next lemma allows us to interpret coincidences in topological terms.
The proof is left to the reader since it is straightforward from the definitions.

\begin{lem}\label{lem:coinc_top}
Let $0 \le n < m$.
Then $k$ belongs to $\coinc'(\tau_{[n,m)})$ if and only if there exists a letter $a \in \cA_n$ such that $S^{k |\tau_{[0,n)}|} B'_m \subseteq B_n(a)$.
\end{lem}

The next lemma is a first step to control the measure of the set of points having a unique $\pimeq$-preimage.  

\begin{lem}\label{lem:measure_cI'n_as_limit}
We have that
\[
\nu(\cI'_{n})
\ge \sup_{m > n} \frac{|\coinc'(\tau_{[n,m)})|}{|\tau_{[n,m)}|}, \enskip
n \ge 0.
\]
\end{lem}

\begin{proof}
Let $n \ge 0$ and define the sets
\[
D_{n, m}
= \{z \in \ZZ_{(|\tau_{[0,n)}|)_{n \ge 0}} : z_m \in z_n + |\tau_{[0,n)}| \coinc'(\tau_{[n,m)})\}, \enskip
m > n.
\]
We first prove that $D_{n,m}$ is contained in $\cI'_n$ for all $m > n$.
Indeed, if $z \in D_{n,m}$ from \Cref{lem:coinc_top} there exists a letter $a \in \cA_n$ such that
\[
S^{z_m} B'_m
\subseteq S^{z_n} B_n(a)
\]
Thus if $x$ is in $\pimeq^{-1}(\{z\}) \cap \TT'_{\ge n}$ then $x$ is in the atom $S^{z_n} B_n(a)$ and thus $z$ is in $\cI'_n$.

From \Cref{lem:measure_cylinders}, we obtain
\[
\nu(\cI'_n)
\ge \sup_{m>n} \nu(D_{n,m})
= \sup_{m> n} \frac{|\coinc' (\tau_{[n,m)})|}{|\tau_{[n,m)}|}.
\]
\end{proof}

As it is more easy to control the positions that are not coincidences, we give the following definition.
For $0 \le n < m$ and letters $a$ and $b$ in $\cA_n$ we define
\[
C_{n,m}(a,b)
= \{k \in [0,|\tau_{[n,m)}|) : S^{k |\tau_{[0,n)}|} B'_m \cap B_n(c) \not= \emptyset, \enskip \forall c \in \{a,b\}\}.
\]
Observe that, since the partition $\PP_m$ is finer than $\PP_n$, for all $k$ in $C_{n,m}(a,b)$ and $c$ in $\{a,b\}$ there exists a letter $c'$ in $\cA'_m$ such that
\begin{equation}
\label{eq:inclusion_C_n_m}
S^{k |\tau_{[0,n)}|} B_m(c')
\subseteq B_n(c).
\end{equation}

The following lemma will allow us define two disjoint sets in $X_{\btau}$ with the same projection under $\pimeq$ and whose $\nu$-measure is controlled in terms of coincidences.

\begin{lem} \label{lem:sets_C_n_Nell}
Let $n \ge 0$.
Then there exist distinct letters $a$ and $b$ in $\cA_n$ and an increasing sequence $(N_\ell)_{\ell \ge 0}$ such that
\begin{equation}\label{eq:estimate_non_coinc}
\frac{|C_{n, N_\ell}(a,b)|}{|\tau_{[n,N_\ell)}|}
\ge \frac{1}{2|\cA_n|^2}
\Bigg(1- \limsup_{\ell \to +\infty}
\frac{|\coinc'(\tau_{[n,N_\ell)})|}{|\tau_{[n,N_\ell)}|} \Bigg), \enskip
\ell \ge 0.
\end{equation}
If in addition the directive sequence $\btau$ is $(\cA'_n)_{n \ge 0}$-adapted to $\mu$, then $(N_\ell)_{\ell \ge 0}$ can be defined so that it satisfies
\begin{equation} \label{eq:estimate_freqs} \frac{|\tau_{[N_\ell, N_{\ell+1})}(c)|_d}{|\tau_{[N_\ell, N_{\ell+1})}|}
\le \frac{1}{2^\ell |\cA_{N_\ell}|} + \mu(\TT_{N_\ell}(d)), \enskip
c \in \cA'_{N_{\ell + 1}}, \enskip
d \in \cA_{N_\ell}, \enskip
\ell \ge 0.
\end{equation}
\end{lem}

\begin{proof}
Let $n \ge 0$ and $n_0 \ge n$ such that
\begin{equation} \label{eq:first_estimate_non_coinc}
\frac{|[0, |\tau_{[n,N)}|) \setminus \coinc'(\tau_{[n,N)})|}{|\tau_{[n, N)}|} \geq
\frac{1}{2}
\Bigg(1- \limsup_{N \to +\infty}
\frac{|\coinc'(\tau_{[n,N)})|}{|\tau_{[n,N)}|} \Bigg), \enskip
N > n_0.
\end{equation}
By definition, for every $N > n_0$ and $k \in [0, |\tau_{[n,N)}|) \setminus \coinc'(\tau_{[n,N)})$ there exist distinct letters $a_{N, k}$ and $b_{N, k}$ in $\cA_n$ and letters $a'_{N, k}$ and $b'_{N, k}$ in $\cA'_N$ such that the $k$-th letter of $\tau_{[n,N)}(a'_{N,k})$ (resp. of $\tau_{[n,N)}(b'_{N,k})$) is $a_{N,k}$ (resp. $b_{N,k}$).
This translates into
\[
S^{k |\tau_{[0,n)}|} B'_N \cap B_n(c) \not= \emptyset, \enskip \forall c \in \{a_{N,k}, b_{N,k}\},
\]
or $k \in C_{n,N}(a_{N,k}, b_{N,k})$.
Now, for each $N > n_0$ we use the Pigeonhole principle to obtain distinct letters $a_N$ and $b_N$ in $\cA_n$ and a subset $C'_{n, N}$ of $C_{n, N}(a_N, b_N)$ with
\[
a_{N,k} = a_N, \enskip
b_{N,k} = b_N, \enskip
k \in C'_{n, N},
\]
\begin{equation} \label{eq:second_estimate_non_coinc}
|C'_{n, N}|
\ge |[0, |\tau_{[n,N)}|) \setminus \coinc'(\tau_{[n,N)})|
/ |\cA_n|^2.
\end{equation}
Moreover, we can find an increasing sequence $(N_\ell)_{\ell\geq0}$ with $N_0 > n_0$ and distinct letters $a,b \in \cA_n$ such that $a = a_{N_\ell}$ and $b = b_{N_\ell}$ for all $\ell\geq0$.
Therefore the sets $C_{n,N_\ell}(a, b)$ satisfy \eqref{eq:estimate_non_coinc} by \eqref{eq:first_estimate_non_coinc} and \eqref{eq:second_estimate_non_coinc}.

\smallskip

If $\btau$ is $(\cA'_n)_{n \ge 0}$-adapted to $\mu$, from \Cref{prop:mu_generic_and_freqs} we can construct $(N_\ell)_{\ell \ge 0}$ inductively so that it additionally satisfies \eqref{eq:estimate_freqs}.
\end{proof}

The next lemma provides two disjoint sets that, roughly speaking, separate the fibers and whose measures can be controlled by the density of coincidences.
It is the main lemma of this section.
Its proof is subdivided in some technical claims.

\begin{lem} \label{lem:limsup_lemma}
Let $n \ge 0$.
There exist sets $E_n$ and $E'_n$ in $X_{\btau}$ such that
\begin{enumerate}
    \item $E_n$ and $E'_n$ are disjoint;
    \item $\pimeq(E_n) = \pimeq(E'_n)$; and
    \[
    \nu(\pimeq(E_n))
    = \nu(\pimeq(E'_n))
    \ge \frac{1}{2|\cA_n|^2}
    \Bigg(1- \limsup_{N \to +\infty}
    \frac{|\coinc'(\tau_{[n,N)})|}{|\tau_{[n,N)}|} \Bigg).
    \]
\end{enumerate}
If in addition the directive sequence $\btau$ is $(\cA'_n)_{n \ge 0}$-adapted to $\mu$ for the constant $\delta > 0$, then
\begin{enumerate}
    \item 
    $\mu(E_n)$ and $\mu(E'_n)$ are greater than
    \[
    \frac{\delta}{2|\cA_n|^2}
    \Bigg(1- \limsup_{N \to +\infty}
    \frac{|\coinc'(\tau_{[n,N)})|}{|\tau_{[n,N)}|} \Bigg); \enskip
    \text{and}
    \]
    \item
    %we have
    $
    \pimeq(E_n \cap U)
    = \pimeq(E'_n \cap U)
    \pmod \nu
    $
    for any Borel set $U$ in $X_{\btau}$ such that $\mu(U) = 1$.
\end{enumerate}
\end{lem}

\begin{proof}
From \Cref{lem:sets_C_n_Nell}, there exist distinct letters $a$ and $b$ in $\cA_n$ and an increasing sequence $(N_\ell)_{\ell \ge 0}$ such that \eqref{eq:estimate_non_coinc} holds.
For $\ell \ge 0$ define $C_{n, N_\ell} = C_{n, N_\ell}(a,b)$ and observe that if $k$ is in $C_{n, N_\ell}$ and $c$ is in $\{a,b\}$, then \eqref{eq:inclusion_C_n_m} implies that there exists a nonempty set $\cA_{n, N_\ell, k}(c) \subseteq \cA'_{N_\ell}$ such that
\[
S^{k |\tau_{[0,n)}|} B_{N_\ell}(c')
\subseteq B_n(c), \enskip
c' \in \cA_{n, N_\ell, k}(c).
\]

Let $K \subseteq X_{\btau}$ be a compact set.
For all $\ell \ge 0$ we set
\begin{align*}
F_{n, N_\ell}
&= \{z \in \ZZ_{(|\tau_{[0,n)}|)_{n \ge 0}}: z_{N_\ell} \in z_n + |\tau_{[0,n)}| C_{n, N_\ell}\};
& \ 
F_{n, \infty}
&= \limsup_{\ell \to +\infty} F_{n, N_\ell}\\
F_{n, N_\ell, K}(c)
&= \{z \in F_{n, N_\ell} : S^{z_{N_\ell}} B'_{N_\ell} \cap S^{z_n} B_n(c) \cap K \not= \emptyset\};
& \ 
c &\in \{a,b\}\\
F_{n, \infty, K}
&= \limsup_{\ell \to +\infty} F_{n, N_\ell, K}(a);\\
F'_{n, \infty, K}
&= \limsup_{\ell \to +\infty} F_{n, N_\ell, K}(b)\\
E_{n, N_\ell}(c)
&= \bigcup_{k \in C_{n,N_\ell}} \bigcup_{\substack{c' \in \cA_{n, N_\ell, k}(c) \\ 0 \le j < |\tau_{[0,n)}|}} S^{k|\tau_{[0,n)}| + j} B_{N_\ell}(c'); & \enskip
c &\in \{a,b\}\\
E_n
&= \limsup_{\ell \to +\infty} E_{n, N_\ell}(a);\\
E'_n
&= \limsup_{\ell \to +\infty} E_{n, N_\ell}(b).
\end{align*}

Observe that, from the definition of the sets $C_{n, N_\ell}$, we have
\[
F_{n, \infty, X_{\btau}}
= F'_{n, \infty, X_{\btau}}
= F_{n, \infty}
\]
and that
\[
E_n \subseteq \TT_n(a) \enskip
\text{and} \enskip
E'_n \subseteq \TT_n(b).
\]
In particular, $E_n$ and $E'_n$ are disjoint.

\begin{claim}\label{cl:F_n_K}
For all compact set $K \subseteq X_{\btau}$ we have
\[
\pimeq(E_n \cap K)
= F_{n, \infty, K} \enskip
\text{and} \enskip
\pimeq(E'_n \cap K)
= F'_{n, \infty, K}.
\]
\end{claim}

\begin{proof}
By symmetry, it is enough to show that $\pimeq(E_n \cap K)
= F_{n, \infty, K}$.
Observe that if $x$ belongs to $E_{n, N_\ell}(a) \cap K$ then $\pimeq(x)$ is in $F_{n, N_\ell, K}(a)$; hence we have the inclusion
\[
\pimeq(E_n \cap K)
\subseteq F_{n, \infty, K}.
\]
Suppose now that $z$ belongs to $F_{n, \infty, K}$.
Then, there is an infinite set $\Lambda \subset \NN$ such that $z$ belongs to $F_{n, N_\ell}$ and a point $x_\ell$ in $S^{z_{N_\ell}} B'_{N_\ell} \cap S^{z_n} B_n(a) \cap K$ for all $\ell \in \Lambda$.
Since $X_{\btau}$ is compact, we can find an infinite set $\Lambda' \subseteq \Lambda$ and $x \in X_{\btau}$ such that $x_\ell \to x$ as $\ell \to +\infty$ and $\ell \in \Lambda'$.
Then $z = \pimeq(x)$ and $x \in E_n \cap K$ since $K$ is compact, which shows that $F_{n, \infty, K}
\subseteq \pimeq(E_n \cap K)$.
\end{proof}

\begin{claim}
We  have $\pimeq (E_n )= \pimeq(E'_n)= F_{n,\infty}$ and
\[
\nu(F_{n, \infty})
\ge \frac{1}{2|\cA_n|^2}
\Bigg(1- \limsup_{N \to +\infty}
\frac{|\coinc'(\tau_{[n,N)})|}{|\tau_{[n,N)}|} \Bigg).
\]
\end{claim}

\begin{proof}
If we put $K = X_{\btau}$ in \Cref{cl:F_n_K} we obtain
$
\pimeq(E_n)
= \pimeq(E'_n)
= F_{n, \infty}$.
On the other hand, from \Cref{lem:measure_cylinders} and \eqref{eq:estimate_non_coinc} we obtain
\[
\nu(F_{n, \infty})
\ge \limsup_{\ell \to +\infty} \frac{|C_{n, N_\ell}|}{|\tau_{[n, N_\ell)}|}
\ge \frac{1}{2|\cA_n|^2}
\Bigg(1- \limsup_{N \to +\infty}
\frac{|\coinc'(\tau_{[n,N)})|}{|\tau_{[n,N)}|} \Bigg).
\]
\end{proof}

This proves the first part of \Cref{lem:limsup_lemma}.

\medskip

We henceforth assume that $\btau$ is $(\cA'_n)_{n \ge 0}$-adapted to $\mu$ for the constant $\delta > 0$.

\begin{claim}
$\mu(E_n)$ and $\mu(E'_n)$ are at least
\[
\frac{\delta}{2|\cA_n|^2}
\Bigg(1 - \limsup_{N \to +\infty}
\frac{|\coinc'(\tau_{[n,N)})|}{|\tau_{[n,N)}|} \Bigg).
\]
\end{claim}

\begin{proof}
Observe that, for each $\ell \ge 0$ and $c \in \{a,b\}$, the sets
\[
\{S^{k |\tau_{[0,n)}| + j} B_{N_\ell}(c') : k \in C_{n, N_\ell}, \enskip 0 \le j < |\tau_{[0,n)}|, \enskip c' \in \cA_{n, N_\ell, k}(c)\}
\]
are disjoint.
Since $\cA_{n, N_\ell, k}(c) \subseteq \cA'_{N_\ell}$, from \eqref{eq:new_delta} we have
\[
\min_{c' \in \cA_{n, N_\ell, k}(c)} \mu(B_{N_\ell}(c'))
\ge \delta / |\tau_{[0, N_\ell)}|.
\]

Thus we obtain 
\[
\mu(E_{n, N_\ell}(c))
\ge |C_{n, N_\ell}| \cdot |\tau_{[0, n)}| \min_{c' \in \cA_{n, N_\ell, k}(c)} \mu(B_{N_\ell}(c'))
\ge \delta \frac{|C_{n, N_\ell}|}{|\tau_{[n, N_\ell)}|}
\]
and \eqref{eq:estimate_non_coinc} implies
\[
\mu\biggl(\limsup_{\ell \to +\infty} E_{n, N_\ell}(c)\biggr)
\ge \frac{\delta}{2|\cA_n|^2}
\Bigg(1 - \limsup_{N \to +\infty}
\frac{|\coinc'(\tau_{[n,N)})|}{|\tau_{[n,N)}|} \Bigg).
\]
\end{proof}

\begin{claim}\label{cl:bound_F_n_K}
For all compact set $K \subseteq X_{\btau}$ we have
\[
\nu\biggl(\limsup_{\ell \to +\infty} F_{n, N_\ell, K}(c)\biggr)
\ge \limsup_{\ell \to +\infty} \nu(F_{n, N_\ell}) - \frac{1}{\delta} \mu(X_{\btau} \setminus K), \enskip
c \in \{a,b\}.
\]
\end{claim}

\begin{proof}
Let $c \in \{a,b\}$ and $\ell \ge 0$.
We define $\cR_{n, N_\ell, K}(c)$ as the set of coordinates $z_{N_\ell}$ in $[0, |\tau_{[0, N_\ell)}|)$ where $z$ ranges over $F_{n, N_\ell} \setminus F_{n, N_\ell, K}(c)$.

From \eqref{eq:haar} we have
\begin{align}\label{eq:ineq_FminusF}
\nu(F_{n, N_\ell} \setminus F_{n, N_\ell, K}(c))
&\le \nu\Bigl(\textstyle\bigcup_{r \in \cR_{n, N_\ell, K}(c)} \{z \in \ZZ_{(|\tau_{[0,n)}|)_{n \ge 0}} :
z_{N_\ell} = r\}\Bigr) \\
&\le \frac{|\cR_{n, N_\ell, K}(c)|}{|\tau_{[0, N_\ell)}|}.
\notag
\end{align}

Observe that if $z$ and $z'$, with $z_{N_l}\not = z_{N_l}'$, belong to $F_{n, N_\ell} \setminus F_{n, N_\ell, K}(c)$, then $S^{z_{N_\ell}} B'_{N_\ell} \cap S^{z_n} B_n(c)$ and $S^{z'_{N_\ell}} B'_{N_\ell} \cap S^{z'_n} B_n(c)$ are disjoint and contained in $X_{\btau} \setminus K$. 
Hence,
\[
|\cR_{n, N_\ell, K}(c)|
\min_{z\in F_{n, N_\ell} \setminus F_{n, N_\ell, K}(c)}
\mu(S^{z_{N_\ell}} B'_{N_\ell} \cap S^{z_n} B_n(c)) \leq
\mu(X_{\btau}\setminus K).
\]
Now, if $z$ belongs to $F_{n, N_\ell} \setminus F_{n, N_\ell, K}(c)$, then from \eqref{eq:inclusion_C_n_m} we have that $S^{z_{N_\ell}} B'_{N_\ell} \cap S^{z_n} B_n(c)$ contains $S^{z_{N_\ell}}B_{N_\ell}(c')$ for some $c'\in\cA'_{N_\ell}$. Hence, by \eqref{eq:new_delta},
\[ 
\min_{z\in F_{n, N_\ell} \setminus F_{n, N_\ell, K}(c)}
\mu(S^{z_{N_\ell}} B'_{N_\ell} \cap S^{z_n} B_n(c)) \geq
\min_{c'\in\cA'_{N_\ell}} \mu(B_{N_\ell}(c')) \geq
\frac{\delta}{|\tau_{[0,N_\ell)}|}.
\]
Combining this with \eqref{eq:ineq_FminusF} yields
\[
\nu(F_{n, N_\ell} \setminus F_{n, N_\ell, K}(c)) \leq
\frac{1}{\delta}\mu(X_{\btau}\setminus K),
\]
and then we finally obtain
\[
\nu\biggl(\limsup_{\ell \to +\infty} F_{n, N_\ell, K}(c)\biggr)
\ge \limsup_{\ell \to +\infty} \nu(F_{n, N_\ell, K}(c))
\ge \limsup_{\ell \to +\infty}
\nu(F_{n, N_\ell}) - \frac{1}{\delta}\mu(X_{\btau} \setminus K).
\]
\end{proof}

\begin{claim}\label{cl:identity_limsup_lim}
We have
\[
\nu(F_{n, \infty})
= \limsup_{\ell \to +\infty} \nu(F_{n, N_\ell}).
\]
\end{claim}

\begin{proof}
Define
\begin{align*}
G_\ell
&= \{k \in [0, |\tau_{[N_\ell, N_{\ell+1})}|) : S^{k |\tau_{[0,N_\ell)}|} B'_{N_{\ell + 1}} \subseteq B'_{N_\ell}\}\\
H_\ell
&= \{z \in \ZZ_{(|\tau_{[0,n)}|)_{n \ge 0}} : S^{z_{N_{\ell+1}}} B'_{N_{\ell+1}} \subseteq S^{z_{N_\ell}} B'_{N_\ell}\}, \enskip
\ell \ge 0\\
H_{\ell^\ast, \infty}
&= \bigcap_{\ell \ge \ell^\ast} H_\ell, \enskip
\ell^\ast \ge 0.
\end{align*}
Observe that $k$ belongs to $G_\ell$ if and only if for all $c$ in $\cA'_{N_{\ell+1}}$ the $k$-th letter of $\tau_{[N_\ell, N_{\ell+1})}(c)$ is in $\cA'_{N_\ell}$.
Hence
\[
|[0, |\tau_{[N_\ell, N_{\ell+1})}|) \setminus G_\ell|
\le \sum_{c \in \cA'_{N_{\ell+1}}} \sum_{d \in \cA_{N_\ell} \setminus \cA'_{N_\ell}} |\tau_{[N_\ell, N_{\ell+1})}(c)|_d.
\]

From \eqref{eq:estimate_freqs} we obtain
\begin{align*}
\frac{|[0, |\tau_{[N_\ell, N_{\ell+1})}|) \setminus G_\ell|}{|\tau_{[N_\ell, N_{\ell+1})}|}
&\le \sum_{c \in \cA'_{N_{\ell+1}}} \sum_{d \in \cA_{N_\ell} \setminus \cA'_{N_\ell}}
\frac{|\tau_{[N_\ell, N_{\ell+1})}(c)|_d}{|\tau_{[N_\ell, N_{\ell+1})}|}\\
&\le \sup_{n \ge 0} |\cA'_n| \bigg(\frac{1}{2^\ell} + \mu(X_{\btau} \setminus \TT'_{N_\ell})\bigg).
\end{align*}

On the other hand, one has $z\in H_\ell$ if and only if $z_{N_{\ell+1}} \in z_{N_\ell} + |\tau_{[0, N_\ell)}| G_\ell$.
From \Cref{lem:measure_cylinders} we have
\[
\nu(H_\ell)
= |G_\ell| / |\tau_{[N_\ell, N_{\ell+1})}|, \enskip
\ell \ge 0.
\]

Therefore, for $\ell^* \ge 0$ one has
\[
\nu \bigg(X_{\btau} \setminus \bigcap_{\ell \ge \ell^\ast} H_\ell \bigg)
\le \sum_{\ell \ge \ell^\ast} \sup_{n \ge 0} |\cA'_n| \bigg(\frac{1}{2^\ell} + \mu(X_{\btau} \setminus \TT'_{N_\ell})\bigg).
\]
Recall that by \eqref{eq:new_delta} we have $\sup_{n \ge 0} |\cA'_n| \le 1 / \delta$, hence by \eqref{eq:new_adapted}
\begin{equation}
\label{eq:intersection_H}
\lim_{\ell^\ast \to +\infty}
\nu \bigg(X_{\btau} \setminus \bigcap_{\ell \ge \ell^\ast} H_\ell \bigg)
= 0.
\end{equation}

\smallskip

Now let $\ell^\ast \le \ell < \ell'$.
Observe that if $z$ belongs to $F_{n, N_{\ell'}} \cap H_{\ell^\ast, \infty}$, then for $c$ in $\{a,b\}$ we have $S^{z_{N_\ell}} B'_{N_\ell} \cap S^{z_n} B_n(c) \subseteq S^{z_{N_{\ell'}}} B'_{N_{\ell'}} \cap S^{z_n} B_n(c)$, which is nonempty.
We deduce that $z$ belongs to $F_{n, N_\ell}$ and so
\[
F_{n, N_{\ell'}} \cap H_{\ell^\ast, \infty} \subseteq F_{n, N_{\ell}} \cap H_{\ell^\ast, \infty},
\]
{\em i.e.}, the sequence of sets $(F_{n, N_\ell} \cap H_{\ell^\ast, \infty})_{\ell \ge \ell^\ast}$ is decreasing.
We obtain
\begin{align*}
& \nu\biggl(\limsup_{\ell \to +\infty} F_{n, N_\ell}\biggr) \\
= & \nu\biggl(\limsup_{\ell \to +\infty} F_{n, N_\ell} \cap H_{\ell^\ast, \infty}\biggr) + \nu\biggl(\limsup_{\ell \to +\infty} F_{n, N_\ell} \cap (X_{\btau} \setminus H_{\ell^\ast, \infty})\biggr)\\
\le & \limsup_{\ell \to +\infty} \nu(F_{n, N_\ell}) + \nu(X_{\btau} \setminus H_{\ell^\ast, \infty}).
\end{align*}
If we let $\ell^\ast \to +\infty$, from \eqref{eq:intersection_H} we deduce
\[
\nu(F_{n, \infty})
= \limsup_{\ell \to +\infty} \nu(F_{n, N_\ell}),
\]
which proves the claim.
\end{proof}

\begin{claim}
We have
\[
\pimeq(E_n \cap U)
= \pimeq(E'_n \cap U)
\pmod \nu
\]
for any Borel set $U \subset X_{\btau}$ such that $\mu(U) = 1$.
\end{claim}

\begin{proof}
By symmetry, it is enough to prove that $\pimeq(E_n \cap U) = F_{n,\infty } \pmod \nu$.
Since $\pimeq(E_n) = F_{n, \infty}$, we have the inclusion $\pimeq(E_n \cap U) \subseteq F_{n, \infty}$.
It remains to prove that $\nu(\pimeq(E_n \cap U)) \ge \nu(F_{n, \infty})$.
Let $\epsilon > 0$.
Since $X_{\btau}$ is compact, $\mu$ is a regular measure and there exists a compact set $K_\epsilon \subseteq U$ with $\mu(X_{\btau} \setminus K_\epsilon) < \epsilon$.

From \Cref{cl:F_n_K} we have
\[
\pimeq(E_n \cap K_\epsilon)
= F_{n, \infty, K_\epsilon}
\]
and from \Cref{cl:bound_F_n_K} and \Cref{cl:identity_limsup_lim} we deduce
\[
\nu(F_{n, \infty, K_\epsilon})
\ge \limsup_{\ell \to +\infty} \nu(F_{n, N_\ell}) - \frac{\epsilon}{\delta}
\ge \nu(F_{n, \infty}) - \frac{\epsilon}{\delta}.
\]
Thus, since $K_\epsilon$ is a subset of $U$, one has
\[
\nu(\pimeq(E_n \cap U))
\ge \nu(F_{n, \infty, K_\epsilon})
\ge \nu(F_{n, \infty}) - \frac{\epsilon}{\delta}.
\]
If we let $\epsilon$ go to $0$, we deduce $\nu(\pimeq(E_n \cap U)) \ge \nu(F_{n, \infty})$, which finishes the proof.
\end{proof}
The achieves the proof of the \Cref{lem:limsup_lemma}.
\end{proof}

%%%%%%%%%%%%%%%%%%%%%%%%%%%%%%%%%%%%%%%%%%%%%%%%%%%%%%%%%%%%
\section{Proof of the main results}
%%%%%%%%%%%%%%%%%%%%%%%%%%%%%%%%%%%%%%%%%%%%%%%%%%%%%%%%%%%%
We will freely use the terminology introduced in \Cref{section:prelim}.

\subsection{Proof of {\Cref{theo:charac_MEF_is_regular}}}

\begin{proof}
From {\Cref{prop:toeplitzSadic}} and recognizability of $\btau$ it is easy to see that \Cref{item:JK_1}, \Cref{item:JK_2}, \Cref{item:MEF_2} and \Cref{item:MEF_3} in \Cref{theo:charac_MEF_is_regular} are equivalent.

We define
\[
a_{n,N}
= 1 - \frac{|\coinc(\tau_{[n,N)})|}{|\tau_{[n,N)}|}, \enskip
0 \le n < N.
\]
From \Cref{lem:inclusion_coinc} we deduce that inequality \eqref{eq:ineq_product} holds.
Hence, we obtain that \Cref{item:MEF_2} and \Cref{item:MEF_4} are equivalent.

\smallskip

It remains to show that \Cref{item:MEF_1} and \Cref{item:MEF_2}.
We set $\cA'_n = \cA_n$ for each $n \ge 0$.
Observe that, with this choice, the directive sequence $\btau$ is $(\cA'_n)_{n \ge 0}$-weakly-adapted to any invariant probability measure $\mu$ of $(X_{\btau}, S)$.
In this case, we have
\[
\cI'
= \{z \in \ZZ_{(|\tau_{[0,n)}|)_{n \ge 0}}:
|\pimeq^{-1}(\{z\})| = 1\}.
\]
Suppose that \Cref{item:MEF_1} holds.
By contradiction, if \Cref{item:MEF_2} does not hold, there exists $n \ge 0$ such that
\[
\limsup_{N \to +\infty}
\frac{|\coinc(\tau_{[n,N)})|}{|\tau_{[n,N)}|}
< 1.
\]
\Cref{lem:limsup_lemma} ensures the existence of disjoints sets $E_n$ and $E'_n$ in $X_{\btau}$ such that $\nu(\pimeq(E_n)) > 0$ and if $z \in \pimeq(E_n)$, then $\pimeq^{-1}(\{z\}) \cap E_n \not= \emptyset$ and $\pimeq^{-1}(\{z\}) \cap E'_n \not= \emptyset$.
Thus $\ZZ_{(|\tau_{[0,n)}|)_{n \ge 0}} \setminus \cI'$ contains $\pimeq(E_n)$, which contradicts the fact that $\nu(\cI') = 1$.

\smallskip

Now assume that \Cref{item:MEF_2} holds.
Then, \Cref{lem:cI'_equals_bigcap_cI'n} and \Cref{lem:measure_cI'n_as_limit} imply that $\nu(\cI') = 1$.
This shows that \Cref{item:MEF_1} holds and finishes the proof.
\end{proof}

\subsection{Proof of \Cref{theo:charac_MEF_is_iso}}

\begin{proof}
Suppose that \Cref{item:iso_1} in \Cref{theo:charac_MEF_is_iso} holds.
This implies that there exists a Borel subset $U$ of $X_{\btau}$ and a Borel subset $V$ of $\ZZ_{(|\tau_{[0,n)}|)_{n \ge 0}}$ such that $\pimeq : U \to V$ is a bijection and $\mu(U) = \nu(V) = 1$.

By contradiction, if \Cref{item:iso_2} does not hold, then there exists $n \ge 0$ such that
\[
\limsup_{N \to +\infty}
\frac{|\coinc_{\cA'_N}(\tau_{[n,N)})|}{|\tau_{[n,N)}|}
< 1.
\]
\Cref{lem:limsup_lemma} then ensures the existence of disjoints sets $E_n$ and $E'_n$ in $X_{\btau}$ such that $\mu(E_n) > 0$, $\mu(E'_n) > 0$ and
\[
\pimeq(E_n \cap U)
= \pimeq(E'_n \cap U)
\pmod \nu.
\]
This implies, as $\pimeq\colon U\to V$ is a bijection, that $E_n \cap U = E'_n \cap U \pmod \mu$, which contradicts the fact that $E_n$ and $E'_n$ are disjoint and of positive measure.

\smallskip

Now assume that \Cref{item:iso_2} holds.
Define
\[
X_\mu
= \pimeq^{-1}(\cI') \cap \TT'.
\]
From \Cref{lem:cI'_equals_bigcap_cI'n}, \Cref{lem:measure_cI'n_as_limit} and \eqref{eq:measure_TT'}, we have $\mu(X_\mu)
= 1$ and $\nu(\pimeq(X_\mu))
= 1$.
On the other hand, the definition of $X_\mu$ implies that $\pimeq : X_\mu \to \pimeq(X_\mu)$ is a bijection.
This shows that \Cref{item:iso_1} holds and finishes the proof.
\end{proof}

\begin{rema}\label{rema:series_rank_two}
Let us observe that when $\btau$ is of alphabet rank two, then it follows from \Cref{lem:inclusion_coinc} and \Cref{item:MEF_4} in \Cref{theo:charac_MEF_is_regular} that $(X_{\btau}, S)$ is a regular extension of its maximal equicontinuous topological factor if and only if
\[
\sum_{n \ge 0} \frac{|\coinc(\tau_n)|}{|\tau_n|}
= +\infty.
\]
Indeed, it follows from \Cref{lem:inclusion_coinc} that if $\btau' = (\tau'_k)_{k \ge 0}$ is a contraction of $\btau$ then
\[
\sum_{k \ge 0} \frac{|\coinc(\tau'_k)|}{|\tau'_k|}
\le \sum_{n \ge 0} \frac{|\coinc(\tau_n)|}{|\tau_n|}.
\]
\end{rema}

%%%%%%%%%%%%%%%%%%%%%%%%%%%%%%
\section{Examples and applications}\label{sec:examples}
%%%%%%%%%%%%%%%%%%%%%%%%%%%%%%

\subsection{Dekking's theorem revisited}
A {\em substitution} is a morphism $\sigma : \cA^\ast \to \cA^\ast$ that is nonerasing.
It naturally defines a directive sequence $\btau = (\sigma, \sigma, \ldots)$ which is of finite alphabet rank and which generates a (possibly empty) substitution subshift $(X_\sigma, S)$.

When $\sigma$ is {\em primitive} (that is, the sequence $\bsigma$ is primitive) the subshift $(X_{\sigma}, S)$ is minimal, has a unique ergodic measure $\mu$ \cite[Chapter V]{Que87}, $\cA_\mu = \cA$ and $\bsigma$ is recognizable \cite{Mos92, Mos96}.

Let $\sigma : \cA^\ast \to \cA^\ast$ be a constant length and primitive substitution.
We say $\sigma$ is {\em pure} if its maximal equicontinuous topological factor is the odometer $(\ZZ_{|\sigma|}, +1)$.
We say $\sigma$ admits a {\em coincidence} if there exists $m \ge 1$ such that $\coinc(\sigma^m) \not= \emptyset$, {\em i.e.}, if the directive sequence $\bsigma = (\sigma, \sigma, \ldots)$ has coincidences.

As a consequence of \Cref{theo:charac_MEF_is_regular} and \Cref{theo:charac_MEF_is_iso}, we recover the following result.

\begin{cor}[{\cite{Dek78}}]
The system $(X_{\sigma}, S, \mu)$ where $\sigma$ is a pure substitution has discrete spectrum if and only if $\sigma$ admits a coincidence.
\end{cor}

\begin{proof}
Suppose that $\sigma$ is pure and admits a coincidence.
Let $m \ge 1$ be such that $\coinc(\sigma^m) \not= \emptyset$.
From \Cref{eq:ineq_coinc_morphisms} we deduce
\[
1 - \frac{|\coinc(\sigma^{mn})|}{|\sigma|^{mn}}
\le \left(1 - \frac{|\coinc(\sigma^m)|}{|\sigma|^m}\right)^n
\to 0 \enskip
\text{as} \enskip
n \to +\infty
\]
which implies by \Cref{theo:charac_MEF_is_regular} that $(X_{\sigma}, S, \mu)$ has discrete spectrum.

Now suppose that $(X_{\sigma}, S, \mu)$ has discrete spectrum.
Every measurable eigenvalue for the system $(X_{\sigma}, S, \mu)$ is necessarily continuous \cite{Hos86}. 
This implies, since $\sigma$ is pure, that there exists a measure theoretical isomorphism $\pi : (X_{\sigma}, S, \mu) \to (\ZZ_{|\sigma|}, +1, \nu)$.
As $(\ZZ_{|\sigma|}, +1, \nu)$ is coalescent (see \Cref{subsec:meq}), the factor map $\pimeq$ is a measure theoretical isomorphism.
Thanks to \Cref{theo:charac_MEF_is_iso} and the fact that $\cA_\mu = \cA$, we deduce
\[
\frac{|\coinc(\sigma^n)|}{|\sigma^n|}
\to 1 \enskip
\text{as} \enskip
n \to +\infty
\]
and hence $\sigma$ has coincidences, finishing the proof.
\end{proof}

We remark that the previous proof shows that for primitive, pure substitutions $\sigma$ the three hypothesis $(X_{\sigma}, S)$ is a regular extension of $(\ZZ_{\sigma}, +1)$; $\pimeq : (X_{\sigma}, S, \mu) \to (\ZZ_{\sigma}, +1, \nu)$ is a measure theoretical isomorphism; and $(X_{\sigma}, S, \mu)$ has discrete spectrum are all equivalent.

\subsection{Example 1}

Rank-one transformations have been studied extensively in ergodic theory since the apparition of the first examples in the '$60$s.
In order to study these systems from the topological and symbolic dynamics perspective, rank-one subshifts have been introduced.
We refer to \cite{Fer96, Fer97} for a more extensive discussion on these subshifts.
Minimal rank-one subshifts are studied from the $\cS$-adic perspective in \cite{AD23}, where the authors prefer to call them {\em Ferenczi subshifts}.
In \cite[Section 4.3]{AD23} it is shown that a minimal Ferenczi subshift possesses an induced system that is (conjugate to) a Toeplitz subshift generated by the constant length directive sequence $\btau = (\tau_n : \cA_{n+1}^\ast \to \cA_n^\ast)_{n \ge 0}$ which have the form    
\begin{equation}\label{eq:Ferenczi}
\tau_n(a)
= L_n a R_n, \enskip a \in \cA_{n+1} 
\end{equation}
for some nonempty words $L_n, R_n$ in $\cA_n^\ast$.
In particular, from \eqref{eq:Ferenczi} we immediately deduce
\[
\frac{|\coinc(\tau_n)|}{|\tau_n|}
\ge \frac{2}{3}, \enskip
n \ge 0
\]
and hence from \Cref{item:MEF_4} in \Cref{theo:charac_MEF_is_regular} we obtain that this induced system is a regular extension of its maximal equicontinuous topological factor.
This improves the remark made in \cite{AD23} saying that these systems are {\em mean equicontinuous} (we refer to \cite{GJY21} for the definition and details about this notion).

\subsection{Example 2}\label{subsec:example_2}
Let $\cA = \{ a,b \}$, $f : \NN \to \NN$ a nonnegative function such that for some pair of constants $s>1$ and $c > 0$ we have $f(n) \ge c n^s$ for $n \ge 0$ and consider the directive sequence $\btau = (\tau_n : \cA^\ast \to \cA^\ast)_{n \ge 0}$ such that for all $n \ge 0$ the morphism $\tau_n : \cA^* \to \cA^*$ is defined by
\[
a \mapsto a^{f(n)} b a, \enskip
b \mapsto b^{f(n)} a a.
\]

Observe that $\btau$ is a constant length primitive directive sequence with coincidences.
Moreover, the sequence $\btau$ is recognizable since the composition matrix of the morphism $\tau_n$ for each $n \ge 0$ is
\[
M_{\tau_n}
= \begin{pmatrix}
f(n)+1 & 2\\
1 & f(n)
\end{pmatrix}
\]
which is invertible \cite[Theorem 4.6]{BSTY19}.
Thus $(X_{\btau}, S)$ is a Toeplitz subshift.

From \Cref{theo:charac_MEF_is_regular} we have that $(X_{\btau}, S)$ is not a regular extension of its maximal equicontinuous topological factor.
Indeed, since $\btau$ is of alphabet rank two, from \Cref{rema:series_rank_two} we have
\[
\sum_{n \ge 0} \frac{|\coinc(\tau_n)|}{|\tau_n|}
= \sum_{n \ge 0} \frac{1}{f(n)+2}
< +\infty
\]
However, the hypothesis on $f$ implies that $(X_{\btau} , S)$ has two ergodic measures $\mu_a$ and $\mu_b$ \cite[Proposition 3.1]{ABKK17}.
We set
\[
\cA_{\mu_a}
= \{a\} \enskip
\text{and} \enskip
\cA_{\mu_b}
= \{b\}.
\]
We see that $\btau$ is $(\cA_{\mu_a})$-adapted to $\mu_a$ and $(\cA_{\mu_b})$-adapted to $\mu_b$ \cite[Theorem 3.3]{BKMS13}.
From \Cref{theo:charac_MEF_is_iso} we conclude that $\pimeq$ defines a measure theoretical isomorphism to the maximal equicontinuous topological factor for both $(X_{\btau}, S, \mu_a)$ and $(X_{\btau}, S, \mu_b)$.

\subsection{Example 3}\label{subsec:example_3}
Let $\cA = \{a, b, c\}$, $f : \NN \to \NN$ a nonnegative function such that for some pair of constants $s>1$ and $c > 0$ we have $f(n) \ge c n^s$ for $n \ge 0$ and consider the directive sequence $\btau = (\tau_n : \cA^\ast \to \cA^\ast)_{n \ge 0}$ such that for all $n \ge 0$ the morphism $\tau_n : \cA^* \to \cA^*$ is defined by
\begin{align*}
a &\mapsto (ab)^{f(n)+1} ac\\ 
b &\mapsto (ab)^{f(n)+1} bc\\
c &\mapsto c(ab)^{f(n)}c cc.
\end{align*}

Observe that $\btau$ is a constant length primitive directive sequence with coincidences.
Moreover, the sequence $\btau$ is recognizable since the composition matrix of the morphism $\tau_n$ for each $n \ge 0$ is
\begin{equation}\label{eq:matrix_ex_3}
M_{\tau_n}
= \begin{pmatrix}
f(n)+2 & f(n)+1 & f(n) \\
f(n)+1 & f(n)+2 & f(n)\\
1 & 1 & 4
\end{pmatrix}
\end{equation}
which is invertible \cite[Theorem 4.6]{BSTY19}.
Thus $(X_{\btau}, S)$ is a Toeplitz subshift.

From \Cref{theo:charac_MEF_is_regular} we have that $(X_{\btau} , S)$ is not a regular extension of its maximal equicontinuous topological factor.
Indeed, by contradiction assume that
\[
1 - \frac{|\coinc(\tau_{[0,N)})|}{|\tau_{[0,N)}|}
\to 0 \enskip
\text{as} \enskip
N \to +\infty.
\]
By induction, it is easy to see that
\[
1 - \frac{|\coinc(\tau_{[0,N+1)})|}{|\tau_{[0,N+1)}|}
\ge \Bigg(1 - \frac{1}{2f(N)+4}\Bigg) \Bigg(1 - \frac{|\coinc(\tau_{[0,N)})|}{|\tau_{[0,N)}|} \Bigg), \enskip
N \ge 1
\]
and hence
\[
1 - \frac{|\coinc(\tau_{[0,N)})|}{|\tau_{[0,N)}|}
\ge \prod_{k=1}^{N-1} \Big(1 - \frac{1}{2f(k)+4}\Big)
\Big( 1 - \frac{|\coinc(\tau_0)|}{|\tau_0|}\Big)
\]
but it is standard to check that, under the hypothesis on $f$, the infinite product $\prod_{k=1}^{\infty} \Big(1 - \frac{1}{2f(k)+4}\Big)$ does not converge to $0$.
This contradiction shows that $(X_{\btau} , S)$ is not a regular extension of its maximal equicontinuous topological factor.

On the other hand, from \eqref{eq:matrix_ex_3} the vectors $(v_n)_{n \ge 0}$ in the statement of \Cref{theo:CriteriaUniqErgo} (for the trivial contraction of $\btau$) satisfy
\[
\sum_{n \ge 0} \frac{|v_n|}{|\tau_n|}
= \sum_{n \ge 0} \frac{2f(n)+1}{2f(n)+4}
= +\infty
\]
which implies that the system $(X_{\btau}, S)$ is uniquely ergodic.
Denote by $\mu$ the unique invariant probability measure and let $\cA_\mu = \{a, b\}$.
We claim that $\btau$ is $(\cA_\mu)$-adapted to $\mu$.
Indeed, from \eqref{eq:inv_measure} and \eqref{eq:measure_tower} we have
\begin{align*}
\mu(\TT_n(a))
&= \mu_n(a) |\tau_{[0,n)}|
= |\tau_{[0,n)}| \sum_{\ell} M_{\tau_n}(a, \ell) \mu_{n+1}(\ell)
\ge f(n) |\tau_{[0,n)}| \sum_{\ell} \mu_{n+1}(\ell)\\
&= f(n) |\tau_{[0,n)}| \sum_{\ell} \frac{\mu(\TT_{n+1}(\ell))}{|\tau_{[0, n+1)}|}
= \frac{f(n)}{2f(n)+4}
\ge \frac{1}{6}.
\end{align*}
A similar computation shows $\mu(\TT_n(b)) \ge 1/6$.
On the other hand, we have
\begin{align*}
\sum_{n \ge 0} \mu(\TT_n(c))
&= \sum_{n \ge 0} \mu_n(c) |\tau_{[0,n)}|
= \sum_{n \ge 0} |\tau_{[0,n)}| \sum_{\ell} M_{\tau_n}(a, \ell) \mu_{n+1}(\ell)\\
&\le \sum_{n \ge 0} 4 |\tau_{[0,n)}| \sum_{\ell} \mu_{n+1}(\ell)
= \sum_{n \ge 0} 4 |\tau_{[0,n)}| \sum_{\ell} \frac{\mu(\TT_{n+1}(\ell))}{|\tau_{[0, n+1)}|}\\
&\le \sum_{n \ge 0} \frac{4}{2 f(n) + 4}
< +\infty
\end{align*}
and thus conditions \eqref{eq:new_adapted} and \eqref{eq:new_delta} are satisfied.

Finally, for all $n \ge 0$ we have
\[
\lim_{N \to +\infty} \frac{|\coinc_\mu(\tau_{[n,N)})|}{|\tau_{[n,N)}|}
= 1.
\]
Indeed, it is easy to check that
\[
1 - \frac{|\coinc_\mu(\tau_{[n,N+1)})|}{|\tau_{[n,N+1)}|}
\ge \frac{1}{2} \Bigg(1 - \frac{|\coinc(\tau_{[n,N)})|}{|\tau_{[n,N)}|} \Bigg), \enskip
N \ge n
\]
and hence
\[
1 - \frac{|\coinc(\tau_{[n,N)})|}{|\tau_{[n,N)}|}
= \frac{1}{2^{N-n-1}}
\Bigg(1 - \frac{|\coinc(\tau_n)|}{|\tau_n|}\Bigg)
\to 0 \enskip
\text{as} \enskip
N \to +\infty.
\]
From \Cref{theo:charac_MEF_is_iso} we conclude that $\pimeq$ defines a measure theoretical isomorphism between $(X_{\btau} , S, \mu)$ and its maximal equicontinuous topological factor.

\subsection{Example 4}
\label{sec:ex4}
Let $\cA = \{a,b,c\}$, $m : \N \to \N$ the function such that $2m(n)+3 = 3^n$ for all $n \ge 0$ and consider the directive sequence $\btau = (\tau_n : \cA^\ast \to \cA^\ast)_{n \ge 0}$ such that for all $n \ge 0$ the morphism $\tau_n : \cA^* \to \cA^*$ is defined by
\begin{align*}
a & \mapsto (ab)^{m(n)} a b c\\
b & \mapsto a(ab)^{m(n)}  a c\\
c & \mapsto (ab)^{m(n)} c b c.
\end{align*}

Observe that $\btau$ is a constant length and primitive directive sequence with coincidences.
Moreover, the sequence $\btau$ is recognizable since the composition matrix of the morphism $\tau_n$ for each $n \ge 0$ is
\[
M_{\tau_n}
= \begin{pmatrix}
m(n)+1 & m(n)+2 & m(n) \\
m(n)+1 & m(n)   & m(n)+1\\
1 & 1 & 2
\end{pmatrix}
\]
which is invertible \cite[Theorem 4.6]{BSTY19}.
Thus $(X_{\btau} , S)$ is a Toeplitz subshift with maximal equicontinuous topological factor $(\ZZ_3, +1)$.
Moreover, from \Cref{theo:CriteriaUniqErgo} it has a unique ergodic measure $\mu$.
Nevertheless $\pimeq$ is not a measure theoretical isomorphism but $(X_{\btau} , S, \mu)$ has discrete spectrum.
Indeed, an easy computation similar to the previous examples which is left to the reader shows that for $\cA_\mu = \{a, b\}$ we have $\btau$ is $(\cA_\mu)$-adapted to $\mu$ and
\[
1 - \frac{|\coinc_\mu(\tau_{[0,N)})|}{|\tau_{[0,N)}|}
= \prod_{k=1}^{N-1} \Big(1 - \frac{2}{3^k}\Big)
\Big( 1 - \frac{|\coinc_\mu(\tau_0)|}{|\tau_0|}\Big),
\]
which does not converge to $0$ as $N \to +\infty$ and from \Cref{theo:charac_MEF_is_iso} we obtain that $\pimeq$ is not a measure theoretical isomorphism.

Let us show it has discrete spectrum. 
Let $P$ be the subset of $X_{\btau}$ given by
\[
P
= \liminf_{n \to +\infty} (\cT_n(a) \cup \cT_n(b)).
\]
The following arguments mainly come from \cite{BDM10}.
The Borel--Cantelli lemma implies that $\mu(P) = 1$.
For $n \ge 0$ define
\[
f_n(x)
= \begin{cases}
$0$ &\text{if \enskip $x \in S^j B_n (a)$, \enskip where $0 \le j < |\tau_{[0,n)}|$ is even}\\
$1$ &\text{if \enskip $x \in S^j B_n(b) \cup S^j B_n(c)$, \enskip where $0 \le j < |\tau_{[0,n)}|$ is odd}.
\end{cases}
\]

Let $A_n = \{x \in X_{\btau} : f_n (x) \not = f_{n+1} (x)\}$.
A quick computation by cases shows that 
\[
A_n
\subseteq \Bigg(\bigcup_{0 \le j < |\tau_{[0,n)}|} S^j B_{n+1}(b)\Bigg) \cup \cT_n (c) \cup \cT_{n+1}(c).
\]
It is easy to check that
\[
\mu(\cT_n(c))
\le \frac{2}{3^n}, \enskip
\mu \Bigg(\bigcup_{0 \le j < |\tau_{[0,n)}|} S^j B_{n+1}(b) \Bigg)
\le \frac{1}{3^n} 
\]
and consequently $\sum \mu(A_n)$ converges.
The Borel--Cantelli lemma again implies that the sequence $(f_n)_{n \ge 0}$ converges $\mu$-almost everywhere to some measurable function $f : X_{\btau} \to \ZZ / 2\ZZ$ that satisfies $f(Sx) = f(x) + 1 \pmod 2$ for $\mu$-almost every $x$.

Denote by $(\ZZ_3 \times \ZZ / 2\ZZ, +(1, 1))$ the product system of $(\ZZ_3, +1)$ and $(\ZZ / 2\ZZ, +1)$ and let $\nu$ be the Haar measure on this product.
Consider the measurable factor map $F : (X_{\btau}, S , \mu) \to (\ZZ_3 \times \ZZ / 2\ZZ, +(1,1), \nu)$ defined by
\begin{align*}
F : X_{\btau}  &\to \ZZ_3 \times \ZZ / 2\ZZ\\
x                &\mapsto  (\pimeq (x), f(x))
\end{align*}

We claim that $F$ defines a measure theoretical isomorphism.
Indeed, it is sufficient to prove that the map $F$ is injective on $P$.
If $F(x) = F(y) = (z, w)$ for some $x,y$ in $P$, we have that for all large values of $n$ we have both $x$ and $y$ belonging to $S^{z_n} \cT_n(w)$, where $z=(z_n)_{n \ge 0}$.
Hence, since $\btau$ is proper, the atoms of the sequence of partitions $(\cT_n)_{n \ge 0}$ generate the topology of $(X_{\btau}, S)$ (see \Cref{subsec:reco}) and we obtain $x = y$.

\subsection{Example 5}
\label{sec:ex5}
Let $\cA = \{1, 2, a, b, c\}$, $s : \NN \to \NN$ a nonnegative function such that $2s(n)+5 = 3^n$ for $n \ge 0$ and consider the directive sequence $\btau = (\tau_n : \cA^\ast \to \cA^\ast)_{n \ge 0}$ such that for all $n \ge 0$ the morphism $\tau_n : \cA^* \to \cA^*$ is defined by
\[
\begin{array}{ll}
1 \mapsto a (12)^{s(n)} 1 2 b c,  & a \mapsto (ab)^{s(n)} a b 1 2 c\\
2 \mapsto a(12)^{s(n)} 2 1 b c, & b \mapsto  a(ab)^{s(n)} b 1 2 c\\
& c \mapsto (ab)^{s(n)} a b 1 2 c.
\end{array}
\]

Observe that $\btau$ is a constant length, recognizable and primitive directive sequence with coincidences.
Thus $(X_{\btau} , S)$ is a Toeplitz subshift with maximal equicontinuous topological factor $(\ZZ_3, +1)$.

The composition matrix of the morphism $\tau_n$ for each $n \ge 0$ is
\[
M_{\tau_n}
= \begin{pmatrix}
s(n)+1 & s(n)+1 & 1 & 1 & 1\\
s(n)+1 & s(n)+1 & 1 & 1 & 1\\
1 & 1 & s(n)+1 & s(n)+1 & s(n)+1\\
1 & 1 & s(n)+1 & s(n)+1 & s(n)+1\\
1 & 1 & 1 & 1 & 1.
\end{pmatrix}
\]

The subshift $(X_{\btau} , S)$ has two ergodic measures $\mu_{a,b}$ and $\mu_{1,2}$.
Indeed, since the rank of each matrix $M_{\tau_n}$ is two, $(X_{\btau} , S)$ has at most two different ergodic measures.
Moreover, the system is not uniquely ergodic thanks to \Cref{theo:CriteriaUniqErgo}: a simple yet tedious computation shows that for any increasing sequence $(n_k)_{k \ge 0}$ the vectors $(v_k)_{k \ge 0}$ satisfy $|v_k| = \prod_{n_k < j \le n_{k+1}} (2 s(j) + 3)$ and so
\[
\sum_{k \ge 0} \frac{|v_k|}{|\tau'_k|}
= \frac{1}{2 s(n_k) + 5}
< +\infty.
\]
where $\tau'_k = \tau'_{[n_k, n_{k+1})}$, $k \ge 0$.

Let $\cA_{\mu_{a, b}} = \{a, b\}$ and $\cA_{\mu_{1, 2}} = \{1, 2\}$.
Observe that, from the form of the composition matrix $M_{\tau_n}$, for any ergodic measure $\mu$ of $(X_{\btau}, S)$ we have
\begin{equation*}
\mu(\cT_n(c)) \to 0, \enskip
\mu(\cT_n(a)) = \mu(\cT_n(b))
\enskip \text{and} \enskip
\mu(\cT_n(1)) = \mu(\cT_n(2)).
\end{equation*}
Thus, without loss of generality, $\btau$ is $(\cA_{\mu_{a, b}})$-adapted to $\mu_{a,b}$ and $(\cA_{\mu_{1,2}})$-adapted to $\mu_{1,2}$.
We left the details to the reader.

The map $\pimeq$ is a measure theoretical isomorphism with respect to $\mu_{1,2}$ whereas it is not with respect to  $\mu_{a,b}$ as it has $-1$ as an eigenvalue.
Moreover, as in the previous example, $(X_{\btau} , S, \mu_{a, b})$ has discrete spectrum: it is measure theoretically isomorphic to $(\ZZ_3 \times \ZZ / 2 \ZZ, +(1, 1), \nu)$ (here $\nu$ is the Haar measure on $\ZZ_3 \times \ZZ / 2\ZZ$).
We leave it to the reader to verify these statements in the light of the calculations made earlier.

\subsection{Example 6}
\label{sec:ex6}
Above we gave example of Toeplitz subshifts with discrete spectrum. 
There are of course many Toeplitz subshifts that has a non discrete spectrum such that those with positive entropy. 
All examples we gave has a finite topological rank and are thus of entropy zero. 
It is interesting to give a finite topological rank Toeplitz shift that does not have a discrete spectrum. 

Let $\cA = \{1, 2\}$, $s : \NN \to \NN$ a nonnegative function such that $3s(n)+1 = 5^{2n}$ for $n \ge 0$ and consider the directive sequence $\btau = (\tau_n : \cA^\ast \to \cA^\ast)_{n \ge 0}$ such that for all $n \ge 0$ the morphism $\tau_n : \cA^* \to \cA^*$ is defined by
\[
1 \mapsto (121)^{s(n)} 2, \enskip
2 \mapsto 1(121)^{s(n)}.
\]
The directive sequence $\btau$ is clearly recognizable. 
Let us show that $(X_{\btau} , S)$ does not have non continuous eigenvalues. 
Suppose that $\lambda = \exp (2i \pi \alpha )$ is such an eigenvalue. 
Then, from  \cite[Proposition 28]{BDM10} one can suppose $\alpha = 1/2$.
But from \cite[Corollary 16]{DFM19} and some computations one can show $\lambda = -1$ cannot be an eigenvalue. 

Thus, the maximal equicontinuous measurable factor is $(\ZZ_5 , +1)$ and it cannot be measure theoretically isomorphic to $(X_{\btau} , S)$ as $(\ZZ_5 , +1)$ is coalescent \cite{HP68}.
Indeed, otherwise $\pimeq$ would be a measure theoretical isomorphism which is not possible 
because $\btau$ does not fulfill \Cref{item:iso_2} of \Cref{theo:charac_MEF_is_iso}.

\subsection{Example 7}
\label{sec:ex7}
We finish this section with the description of a non Toeplitz subshift generated by a constant length directive sequence that has positive topological entropy and where the conclusion of \Cref{theo:charac_MEF_is_iso} holds.

Let $(\ell_n)_{n \ge 0}$ be a sequence of nonnegative numbers such that
\[
\ell_{n+1}
= (\ell_n - 1)! + 1, \enskip
n \ge 0.
\]

For $n \ge 0$ let $\cA_n = \{a_n, b_n(1), \ldots, b_n(\ell_n - 1)\}$.
The choice of $\ell_{n+1}$ implies that there exists a bijection $\pi_n$ between $\{b_{n+1}(1), \ldots, b_{n+1}(\ell_{n+1}-1)\}$ and the set of all words in $\cA_n^\ast$, which can be defined using each letter in $\{b_n(1), \ldots, b_n(\ell_n-1)\}$ exactly once.
Let $\pi_n'$ be an onto map between $\{b_{n+1}(1), \ldots, b_{n+1}(\ell_{n+1}-1)\}$ and $\cA_n^2$.
Consider the directive sequence $\btau = (\tau_n : \cA_{n+1}^\ast \to \cA_n^\ast)_{n \ge 0}$ such that for all $n \ge 0$ the morphism $\tau_n : \cA_{n+1}^\ast \to \cA_n^\ast$ is defined by
\begin{align*}
\tau_n(b_{n+1}(i))
&= \pi_n(b_{n+1}(i)) \pi_n'(b_{n+1}(i)) a_n, \enskip
1 \le i < \ell_{n+1}\\
\tau_n(a_{n+1})
&= a_n^{\ell_n + 1} b_n(1).
\end{align*}

The fact that $\btau$ is a constant length directive sequence that is primitive, injective on letters, recognizable and without coincidences is left to the reader.
Since every word in $\cA_n^2$ occurs in the images of $\tau_n$, it can be checked that the maximal equicontinuous topological factor of $(X_{\btau}, S)$ corresponds to the odometer $(\ZZ_{(|\tau_{[0,n)}|)_{n \ge 0}}, +1)$.
Thus, since $\btau$ does not have coincidences $(X_{\btau}, S)$ is a minimal non Toeplitz subshift.

Stirling's approximation implies that
\[
\log \ell_{n+1}
> (\ell_n - 1) \log(\ell_n - 1) - (\ell_n - 1).
\]
Since $\ell_k - 1 \ge 2^k$ for $k \ge 0$, inductively we obtain
\begin{align*}
\frac{\log \ell_{n+1}}{(\ell_0 + 2) \ldots (\ell_n + 2)}
&> \prod_{k=0}^n \Big(1 - \frac{3}{\ell_k + 2} \Big) \log(\ell_0 - 1) - \sum_{k=0}^n \frac{(\ell_k - 1)\ldots (\ell_n - 1)}{(\ell_0 + 2) \ldots (\ell_n + 2)}\\
&> \prod_{k=0}^n \Big(1 - \frac{3}{2^k} \Big) \log(\ell_0 - 1) - 2.
\end{align*}
From \cite[Lemma 2.6]{BH94} the topological entropy of $(X_{\btau}, S)$ is
\[
\lim_{n \to +\infty} \frac{\log |\cA_n|}{|\tau_{[0,n)}|}
=
\lim_{n \to +\infty} \frac{\log \ell_n}{(\ell_0+2) \ldots (\ell_{n-1} + 2)}
> 0
\]
if we choose $\ell_0$ such that $\prod_{k=0}^\infty \Big(1 - \frac{3}{2^k} \Big) \log(\ell_0 - 1)
> 2$.
In particular $(X_{\btau}, S)$ is not of finite topological rank.

Finally if $\cA'_n = \{a_n\}$ for $n \ge 0$ then there exists an ergodic measure $\mu$ such that $\btau$ is $(\cA'_n)_{n \ge 0}$-adapted to $\mu$.
\Cref{theo:charac_MEF_is_iso} implies that $(X_{\btau}, S, \mu)$ is isomorphic to $(\ZZ_{(|\tau_{[0,n)}|)_{n \ge 0}}, +1, \nu)$.

%%%%%%%%%%%%%%%%%%%%%%%%%%%%%%%%%%%%%%%%%%%%%%%
% Bibliography
%%%%%%%%%%%%%%%%%%%%%%%%%%%%%%%%%%%%%%%%%%%%%%%

\bibliographystyle{alpha}
\bibliography{biblio}

\end{document}